\documentclass[11pt,reqno]{amsart}
\usepackage{amsmath, amsfonts, color, amsthm, graphics, amssymb,url}
\usepackage{stmaryrd}
\usepackage[notcite,notref,final]{showkeys}

\usepackage{cleveref}

\crefname{section}{Section}{Sections}
\crefformat{section}{#2Section~#1#3} 
\Crefformat{section}{#2Section~#1#3} 

\crefname{subsection}{}{Subsections}
\crefformat{subsection}{\S#2#1#3} 
\Crefformat{subsection}{\S#2#1#3}

\crefname{definition}{Definition}{Definitions}
\crefformat{definition}{#2Definition~#1#3} 
\Crefformat{definition}{#2Definition~#1#3} 

\crefname{example}{Example}{Examples}
\crefformat{example}{#2Example~#1#3} 
\Crefformat{example}{#2Example~#1#3} 

\crefname{examplenodiamond}{Example}{Examples}
\crefformat{examplenodiamond}{#2Example~#1#3} 
\Crefformat{examplenodiamond}{#2Example~#1#3} 

\crefname{remark}{Remark}{Remarks}
\crefformat{remark}{#2Remark~#1#3} 
\Crefformat{remark}{#2Remark~#1#3} 

\crefname{remarknodiamond}{Remark}{Remarks}
\crefformat{remarknodiamond}{#2Remark~#1#3} 
\Crefformat{remarknodiamond}{#2Remark~#1#3} 

\crefname{convention}{Convention}{Conventions}
\crefformat{convention}{#2Convention~#1#3} 
\Crefformat{convention}{#2Convention~#1#3} 

\crefname{lemma}{Lemma}{Lemmas}
\crefformat{lemma}{#2Lemma~#1#3} 
\Crefformat{lemma}{#2Lemma~#1#3} 

\crefname{definition-lemma}{Definition-Lemma}{Definition-Lemmas}
\crefformat{definition-lemma}{#2Definition-Lemma~#1#3} 
\Crefformat{definition-lemma}{#2Definition-Lemma~#1#3} 

\crefname{proposition}{Proposition}{Propositions}
\crefformat{proposition}{#2Proposition~#1#3} 
\Crefformat{proposition}{#2Proposition~#1#3} 

\crefname{corollary}{Corollary}{Corollaries}
\crefformat{corollary}{#2Corollary~#1#3} 
\Crefformat{corollary}{#2Corollary~#1#3} 

\crefname{theorem}{Theorem}{Theorems}
\crefformat{theorem}{#2Theorem~#1#3} 
\Crefformat{theorem}{#2Theorem~#1#3} 

\crefname{assumption}{Assumption}{Assumptions}
\crefformat{assumption}{#2Assumption~#1#3} 
\Crefformat{assumption}{#2Assumption~#1#3}

\crefname{notation}{Notation}{Notation}
\crefformat{notation}{#2Notation~#1#3} 
\Crefformat{notation}{#2Notation~#1#3}

\crefname{hypothesis}{Hypothesis}{Hypotheses}
\crefformat{hypothesis}{#2Hypothesis~#1#3} 
\Crefformat{hypothesis}{#2Hypothesis~#1#3}

\crefname{equation}{}{}
\crefformat{equation}{(#2#1#3)} 
\Crefformat{equation}{(#2#1#3)}

\crefname{align}{}{}
\crefformat{align}{(#2#1#3)} 
\Crefformat{align}{(#2#1#3)}

\crefname{proofstep}{Step}{Steps}
\crefformat{proofstep}{#2Step~#1#3} 
\Crefformat{proofstep}{#2Step~#1#3}

\usepackage{tikz}
\usetikzlibrary{arrows,calc,decorations.pathreplacing,decorations.markings,intersections,shapes.geometric,through,fit,shapes.symbols,positioning,decorations.pathmorphing}

\numberwithin{equation}{section}
 \theoremstyle{plain}
\newtheorem{theorem}[equation]{Theorem}

\newtheorem{definition-prop}[equation]{Definition-Proposition}
\newtheorem{lemma}[equation]{Lemma}

\newtheorem{corollary}[equation]{Corollary}
\newtheorem{proposition}[equation]{Proposition}

\theoremstyle{definition}
\newtheorem{definition}[equation]{Definition}
\newtheorem{notation}[equation]{Notation}

\newtheorem{remark}[equation]{Remark}

\newtheorem{hypothesis}[equation]{Hypothesis}
\newtheorem{example}[equation]{Example}

\newcommand{\mc}{\mathcal}

\newcommand{\A}{\mathbb{A}}
\newcommand{\B}{\mathbb{B}}
\newcommand{\I}{\mathbb{I}}

\newcommand{\Hom}{\textrm{Hom}}
\newcommand{\Ext}{\textrm{Ext}}

\newcommand{\ts}{\textstyle}

\newcommand{\gldim}{\textrm{gldim}}
\newcommand{\id}{\textrm{id}}
\newcommand{\ev}{\textrm{ev}}
\newcommand{\kk}{\Bbbk}

\newcommand\bA{\mathbb A}
\newcommand\bB{\mathbb B}

\newcommand\bE{\mathbb E}
\newcommand\bF{\mathbb F}

\newcommand\bI{\mathbb I}

\newcommand\bP{\mathbb P}
\newcommand\bQ{\mathbb Q}

\newcommand\bU{\mathbb U}

\newcommand\cB{\mathcal B}

\newcommand\cF{\mathcal F}
\newcommand\cG{\mathcal G}
\newcommand\cH{\mathcal H}

\newcommand\cO{\mathcal O}

\newcommand\cS{\mathcal S}

\newcommand\cW{\mathcal W}

\newcommand\e{{\sf e}}
\newcommand\f{{\sf f}}

\input xy
\xyoption{all}

\begin{document}

\title[On quantum groups associated to a pair of preregular forms]
{On quantum groups associated to a pair\\ of preregular forms}

\author{Alexandru Chirvasitu}
\address{Chirvasitu: Dept. of Mathematics, University of Washington, Seattle, Washington 98195, USA}
\email{chirva@uw.edu }

\author{Chelsea Walton}
\address{Walton: Dept. of Mathematics, Temple University, Philadelphia, Pennsylvania 19122, USA}
\email{notlaw@temple.edu}

\author{Xingting Wang}
\address{Wang: Dept. of Mathematics, Temple University, Philadelphia, Pennsylvania 19122, USA}
\email{xingting.wang@temple.edu}

  \begin{abstract}
We define the universal quantum group $\mc{H}$ that preserves a pair of Hopf comodule maps, whose underlying vector space maps are preregular forms defined on dual vector spaces. This generalizes the construction of Bichon and Dubois-Violette (2013), where the target of these comodule maps are the ground field. We also recover the quantum groups introduced by Dubois-Violette and Launer (1990), by Takeuchi (1990), by Artin, Schelter, and Tate (1991), and by Mrozinski (2014), via our construction. As a consequence, we obtain an explicit presentation of a universal quantum group that coacts simultaneously on a pair of $N$-Koszul Artin-Schelter regular algebras with arbitrary quantum determinant.
 \end{abstract}
 


 \maketitle
 
 \tableofcontents 
 
 \bibliographystyle{abbrv}  
 \section{Introduction}
 Take  $\kk$ to be our base field and let  an unadorned $\otimes$ denote $\otimes_\kk$. All algebraic structures are over $\kk$. 
Our motivation for this work is to determine when a Hopf algebra, that coacts universally on an $N$-Koszul Artin-Schelter regular algebra  $A$ of dimension $d \geq 3$, shares the same homological and ring-theoretic behavior of the comodule algebra $A$. 
Recall that an {\it Artin-Schelter \textnormal{(}AS\textnormal{)}-regular algebra $A$ of dimension $d$} is a (noncommutative) homological analogue of a commutative polynomial ring;  it is, by definition, a connected $\mathbb{N}$-graded algebra $A$  that is {\it AS-Gorenstein} (has injective dimension $d < \infty$ on both sides with $\Ext^i_A(_A\kk, {}_AA) \cong \Ext^i_A(\kk_A,A_A)= \delta_{i,d} ~\kk$), and has global dimension $d$ on both sides. We refer the reader to \cite{Berger:Koszulity} for a discussion of $N$-Koszulity.

\medbreak This study is a continuation of the last two authors' work \cite{Wsquared},  where Hopf algebras that coact universally  on ($N$-Koszul) AS-regular algebras of dimension $2$ were examined. Like in \cite{Wsquared}, here we must consider a nice Hopf quotient of the universal quantum group that coacts on a higher dimensional Artin-Schelter regular algebra to achieve our goal (cf. \cite[Introduction ($\star\star$)]{Wsquared}). For instance, the universal quantum group that coacts on a commutative polynomial ring in two variables $\kk[u,v]$ is too large to reflect the homological and ring-theoretic behavior  of $\kk[u,v]$ \cite{RvdB2}.\footnote{There are fruitful representation-theoretic directions to pursue in this case; see, for example, recent work of Raedschelders and van den Bergh \cite{RvdB1}.}
To remedy this issue, we consider the universal quantum group $\cO_{A,A'}$ that coacts on {\it simultaneously} on a pair of $N$-Koszul Artin-Schelter regular algebras $A, A'$ (\cref{OA(e)A(f)}), which is a Hopf
quotient of each of Manin's universal quantum groups $\cO_A$ of $A$ and  $\cO_{A'}$ of $A'$ \cite{Manin:QGNCG} (\cref{OA(e) OA(f)}). (See \cite[Section~2.2]{CWZ:Nakayama} for a discussion of how one extends the definition of Manin's universal quantum group to the non-quadratic case.)

\medbreak We proceed by considering Dubois-Violette's framework of {\it preregular forms} (Definition~\ref{preregular}) \cite{DV2005, DV2007} to obtain a presentation of the universal quantum groups $\cO_A$, $\cO_{A'}$, and $\cO_{A,A'}$, where the coactions on $A, A'$ have arbitrary quantum determinant, or more precisely, arbitrary {\it homological codeterminant} (\cref{hcodet}). Namely, since $A, A'$ are $N$-Koszul AS-regular, the algebras are realized as {\it superpotential algebras} (\cref{superpotential alg}) associated to {\it preregular} $m$-linear forms $ \e ,  \f $, respectively (Proposition~\ref{propA}). (In particular, these forms are also realized as duals  of {\it twisted superpotentials} in some sense.) See \cref{sec:prereg} for further details. 

\medbreak The reader may wish to refer to Notation~\ref{not} at this point. Now the universal quantum groups of interest in this work are the following.

\begin{theorem}[Definitions~\ref{D:A(e)},~\ref{D:A(f)},~\ref{def:Hef} and Propositions~\ref{P:H(e)},~\ref{P:UH(e)},~\ref{P:H(f)},~\ref{P:UH(f)},~\ref{P:H(ef)},~\ref{P:UHef}] \label{T:intro1} \textcolor{white}{,}\\ Let $ \e $ and $ \f $ be preregular $m$-linear  forms on $n$-dimensional vector spaces $V$, $V^*$, respectively.
\begin{enumerate}
\item Take $\cH( \e )$ \textnormal{(}resp., $\cH( \f )$\textnormal{)} to be the universal Hopf algebra that preserves any right \textnormal{(}resp., left\textnormal{)} Hopf comodule map $\hat{\e}$ from $V^{\otimes m}$ \textnormal{(}resp., $\hat{\f}$ from $V^{* \otimes m}$\textnormal{)} to a one-dimensional comodule, whose underlying vector space map is $ \e $ \textnormal{(}resp., is $ \f $\textnormal{)}.  Then, $\cH( \e )$ \textnormal{(}resp., $\cH( \f )$\textnormal{)} has a finite presentation; in particular, its underlying $\kk$-algebra structure has $2n^2 + 2$ generators and finitely many relations. 

\medskip

\item Take $\cH( \e , \f )$ to be the universal Hopf algebra that preserves any right Hopf comodule map from $V^{\otimes m}$ to a one-dimensional comodule, and also the induced left Hopf comodule map from $V^{* \otimes m}$ to a one-dimensional comodule, whose underlying vector space maps are $ \e $ and~$\f $, respectively. Then, $\cH( \e, \f )$ has a finite presentation; in particular, its underlying $\kk$-algebra structure has $n^2 + 4$ generators and finitely many relations. 
\end{enumerate}
\end{theorem}

 We also provide a description of the quantum group in part (b) as a pushout of the quantum groups in part (a); see Theorem~\ref{push}. 

\medbreak Next, we consider coactions on pairs of $N$-Koszul AS-regular algebras. See Definition~\ref{D:Sym} for the notion of a {\it balanced} simultaneous coaction on a pair of such algebras.

\begin{theorem} \label{T:intro2} Retain the notation of Theorem~\ref{T:intro1}. Take $N$-Koszul Artin-Schelter algebras \linebreak $A=A( \e ,N)$ and $A'=A( \f ,N)$ associated to preregular multilinear forms $ \e $ and $ \f $ that are generated as algebras by $V^*$ and $V$, respectively. Then,
\begin{enumerate}
\item[(a)] \textnormal{[}Proposition~\ref{HeHf-coacts}, \cref{pr.univ}\textnormal{]} $\cH( \e )$  and $\cH( \f )$ coact universally on $A$ and $A'$ from the left and right, respectively, that is, $\cO_A = \cH( \e )$ and $\cO_{A'} = \cH( \f )$ as Hopf algebras; and 
\medskip

\item[(b)] \textnormal{[}\cref{Hef-coacts}(a), \cref{H-Oisom}(a)\textnormal{]}  $\cH( \e , \f )$ coacts universally on $A$ and $A'$  from the left and right in a balanced fashion, that is, $\cO_{A,A'} = \cH( \e ,  \f )$ as Hopf algebras.
\end{enumerate}
\end{theorem}

We recover several quantum groups in the literature as special cases of our Hopf algebra $\cH( \e ,  \f )$, including quantum groups of  Artin, Schelter, and Tate \cite{AST}, of Bichon and Dubois-Violette \cite{BDV}, of Dubois-Violette and Launer \cite{DVL}, and of Takeuchi \cite{Takeuchi}. Naturally, the commutative Hopf algebra $\cO(GL_n)$ also arises in our construction; this is discussed in \cref{polynomial}. See \cref{oldexamples} for details. 

\medbreak In fact, $\cH( \e ,  \f )$ is a multilinear version of Mrozinski's $GL(2)$-like quantum group \cite{Mrozinski} for generic $ \e $ and $ \f $; in the generic case, we denote $\cH$ by $\cG( \e ,  \f )$ and obtain a $GL(n)$-like quantum group (using Corollary~\ref{C:central} with Notation~\ref{N:odot} in Definition~\ref{D:G(e,f)}). 

\medbreak We also obtain results similar to Theorem~\ref{T:intro2} for $SL(n)$-like quantum groups, which we  denote by $\cS(\e)$, $\cS(\f)$ and $\cS( \e , \f )$ (Definitions~\ref{S(e) S(f)} and \ref{D:S(e,f)}); namely, see Definitions~\ref{OA(e) OA(f)}(b,d) and~\ref{OA(e)A(f)}(b), along with Proposition~\ref{Hef-coacts}(b) and Theorems~\ref{pr.univ} and~\ref{H-Oisom}(b). These universal quantum groups (under a compatibility condition for $\cS( \e , \f )$) are {\it cosovereign} due to Theorem~\ref{pr.cosv_SeSf} (cf. \cite[Section~5]{cosv}). It means that there exists a monoidal natural isomorphism between the identity and the double dual functor on the category of their finite-dimensional (right or left) modules. These $SL(n)$-like quantum groups, it turns out, form a natural setting for defining and studying cyclic cohomology \cite{CM,crainic}. 

\medbreak Returning to our original goal, we now have a candidate (namely,  $\cH( \e , \f )$) for a universal quantum group that reflects the behavior of the $N$-Koszul Artin-Schelter regular algebras upon which it coacts. We provide results on:
\medskip

\begin{itemize}
\item how to compute the homological codeterminants of the $\cH( \e )$-coaction on $A=A( \e )$ and of $\cH( \f )$-coaction on $A'=A( \f )$ (and thus of the $\cH( \e , \f )$-coactions on $A$, $A'$) (\cref{thm:codet});
\medskip

\item when the homological codeterminants of the $\cH( \e , \f )$-coactions on both $A$ and $A'$ are central (Corollary~\ref{C:central}(c)); and 
\medskip

\item  when the Hopf algebra $\cH( \e , \f )$ is involutory (\cref{Hef-invol}).
\end{itemize}
\medskip

The first result is a generalization of \cite[Theorem~2.1]{CKWZ}, the $d=2$ case.  We also believe that there is a connection between Jing-Zhang's study of {\it quantum hyperdeterminants} \cite{JingZhang} and the homological codeterminants used in this work; this is worth further investigation.

\medbreak Finally, by using $\cH( \e , \f )$,  we present new examples of universal quantum groups  that coact on centain $N$-Koszul Artin-Schelter regular algebras of global dimension $d \geq 3$ in \cref{sec:examples}. In this direction, connections between our work and the L\"{u}-Mao-Zhang's \cite{LMZ} study of finite-dimensional Hopf actions of AS-regular algebras of dimension 3 is an avenue for future research.

  
\section{Preregular forms, twisted superpotentials, and $N$-Koszulity} \label{sec:prereg}
 
Let $m \geq N \geq 2$ be integers. Let $Z$ be a finite-dimensional $\kk$-vector space and define the linear map $c: Z^{\otimes m} \to Z^{\otimes m}$ by
 $$ c(z_1 \otimes z_2 \otimes \cdots \otimes z_{m-1} \otimes z_m) := 
 z_m \otimes z_1 \otimes \cdots \otimes z_{m-2} \otimes z_{m-1}.$$
 for any $z_1,\dots,z_m\in Z$.

 \begin{definition}[preregular form] \label{preregular}
We say that an $m$-linear form ${\sf t}: Z^{\otimes m}\to \kk$ on $Z$ is {\it preregular} if it satisfies the following conditions.
\begin{itemize}
\item[(a)] If ${\sf t}(z_1,z_2,\dots,z_m)=0$ for any $z_2,\dots,z_{m}\in Z$, then $z_1=0$.
\item[(b)] ${\sf t}$ is $\phi$-cyclic for some $\phi \in GL(Z)$; that is
\[
\begin{array}{rl}
{\sf t}(z_1, \dots, z_{m-1},z_m) & =\left({\sf t}\circ c\circ (\id^{\otimes(m-1)} \otimes \phi)\right)(z_1 \otimes \cdots \otimes z_{m-1} \otimes z_m)\\
&  =~ {\sf t}\left(\phi(z_m), z_1, \dots, z_{m-1}\right)
\end{array}
\]
for any $z_1, \dots, z_{m} \in Z$. In this case, we  write ${\sf t}=({\sf t}, \phi)$.
\end{itemize}
We say that ${\sf t}$ is {\it cyclic} if (b) holds for some $\phi$.
\end{definition}

The dual notion of a preregular form is a {\it twisted superpotential}, as defined in Mori-Smith \cite{MoriSmith}. 

\begin{definition}[twisted superpotential] \label{twist-superpot}
Take ${\sf s} \in Z^{\otimes m}$ and $\phi \in GL(Z)$. We call ${\sf s}$ a 
\begin{enumerate}
\item {\it superpotential} if $c({\sf s}) = {\sf s}$; 
\item {\it $\phi$-twisted superpotential} if $(\phi \otimes \id^{\otimes(m-1)})c({\sf s}) = {\sf s}$; and 
\item {\it twisted superpotential} if it is a $\phi$-twisted superpotential for some $\phi$; here, we write ${\sf s}=({\sf s}, \phi)$.
\end{enumerate}
\end{definition}

With the identification $Z^{\otimes m} \cong ((Z^*)^{\otimes m})^*$, we have that the set of $\phi$-cyclic  $m$-linear forms on $Z^*$ corresponds bijectively to the set of $\phi$-twisted superpotentials in $Z^{\otimes m}$, as will see in \cref{super-prereg} below. Now consider the following algebra.

\begin{definition}[superpotential algebra $A({\sf s},N)$] \label{superpotential alg}
Given a twisted superpotential ${\sf s} \in Z^{\otimes m}$, the {\it superpotential algebra} associated to ${\sf s}$ is 
$$A({\sf s}, N)  = T Z/ \partial^{m-N}(\kk{\sf s}),$$
where $TZ$ is the tensor algebra on $Z$ and  $\partial(\kk{\sf s}) = \{(\nu \otimes \id^{\otimes(m-1)})(\alpha {\sf s}) ~|~ \nu \in Z^*, \alpha \in \kk\}$ and $\partial^{i+1}(\kk{\sf s}) = \partial(\partial^{i}(\kk{\sf s}))$ for  all $i \geq 0$. Observe that $\partial^{m-N}(\kk{\sf s}) \subseteq  Z^{\otimes N}$.  
\end{definition}

We also define an algebra associated to a cyclic form by identifying it with a superpotential algebra as we will see below. The following lemma is clear. 

\begin{lemma}[${\sf s}^*$, ${\sf t}^*$] \label{super-prereg} Take $Z = \bigoplus_{i=1}^n \kk z_i$, an $n$-dimensional $\kk$-vector space. Then, the results below hold.
\begin{enumerate}
\item Given a $\phi$-twisted superpotential ${\sf s} \in Z^{\otimes m}$, we have that ${\sf s}^* \in ((Z^*)^{ \otimes m})^*$ with ${\sf s}^*(z_{i_1}^*, \dots, z_{i_m}^*) $ equal to the coefficient of $z_{i_1}\cdots z_{i_m}$ in ${\sf s}$ is a $\phi$-cyclic $m$-linear form \textnormal{(}corresponding to ${\sf s}$\textnormal{)}.

\smallskip

\item Given a $\phi$-cyclic $m$-linear form ${\sf t} \in ((Z^*)^{ \otimes m})^*$, we have that  ${\sf t}^* \in Z^{\otimes m}$ given by \linebreak ${\sf t}^* = \sum_{i_1, \dots, i_m=1}^n {\sf t}(z_{i_1}^*, \dots, z_{i_m}^*)z_{i_1}\cdots z_{i_m}$ is a $\phi$-twisted superpotential \textnormal{(}corresponding to ${\sf t}$\textnormal{)}.

\smallskip

\item We obtain that $({\sf s}^*)^* = {\sf s}$ and $({\sf t}^*)^* = {\sf t}$.

\smallskip

\item Given a $\phi$-cyclic $m$-linear form ${\sf t} \in ((Z^*)^{\otimes m})^*$, let $A({\sf t}, N):=TZ/\partial^{m-N}(\kk {\sf t}^*)$, where the relation space $\partial^{m-N}(\kk {\sf t}^*)$ is generated by 
$$\textstyle \sum_{j_1, \dots, j_N = 1}^n ~{\sf t}(z^*_{i_1}, \dots, z^*_{i_r}, z^*_{j_1}, \dots, z^*_{j_N}, z^*_{k_1}, \dots, z^*_{k_s}) z_{j_1} \cdots z_{j_N}$$
for any fixed integers $r,s$ satisfying $r+s=m-N$ and all $i_1, \dots, i_r,k_1,\dots,k_s \in \{1, \dots, n\}$. \qed
\end{enumerate} 
\end{lemma}

The one-to-one correspondence between $\phi$-cyclic forms and $\phi$-twisted superpotentials stated in lemma above actually descends to preregular forms through the following definition.

\begin{definition}[preregular superpotential]
We say that a $\phi$-twisted superpotential ${\sf s} \in Z^{\otimes m}$ is {\it preregular} if it satisfies one of the following equivalent conditions.
\begin{enumerate}
\item ${\sf s}^*$ is a $\phi$-cyclic preregular form on $Z^*$.
\item $(\nu \otimes \id^{\otimes m-1})({\sf s})=0$ for some $\nu\in Z^*$ implies that $\nu=0$.
\end{enumerate}
\end{definition}

Let us collect some examples and properties of superpotential algebras (associated to preregular superpotentials); more details about the following will be discussed in \cref{polynomial} and Section~\ref{Skly3}.

 \begin{example} \label{exSabc} Take $Z = \kk z_1 \oplus \kk z_2 \oplus \kk z_3$ with $m = 3$ and $N=2$.  \begin{enumerate}
 \item We get $A({\sf s}, 2) = \kk[z_1, z_2, z_3]$, a commutative polynomial algebra, by taking
$${\sf s} = z_1z_2z_3 + z_2z_3z_1 + z_3z_1z_2 - z_1z_3z_2 - z_3z_2z_1 - z_2z_1z_3,$$ 
as  $(z_i^* \otimes \id \otimes \id)({\sf s}) = z_{i+1}z_{i+2}-z_{i+2}z_{i+1}$, for $i =1,2,3$ with indices taken modulo 3. In this case ${\sf s}^*: Z^{*\otimes 3} \to \kk$ is given by 
{\small \[
{\sf s}^*(z^*_{i_1}, z^*_{i_2}, z^*_{i_3}) = 
\begin{cases} 
1 & \text{ for } (i_1, i_2, i_3) = (1,2,3), (2,3,1), (3,1,2)\\
-1 & \text{ for } (i_1, i_2, i_3) = (1,3,2), (3,2,1), (2,1,3)\\
0 & \text{ otherwise.}
\end{cases}
\]}
\smallskip

 \item We get $A({\sf s}_{abc}, 2) = Skly_3(a,b,c)$, a  three-dimensional Sklyanin algebra, by taking
 $${\sf s}_{abc} = a(z_1z_2z_3 + z_2z_3z_1 + z_3z_1z_2) + b(z_1z_3z_2 + z_3z_2z_1 + z_2z_1z_3) + c(z_1^3 + z_2^3 + z_3^3)$$ 
 for $[a:b:c] \in \mathbb{P}^2_{\kk}$ with $abc \neq 0$ and $(3abc)^3 \neq (a^3+b^3+c^3)^3$, since $$(z_i^* \otimes \id \otimes \id)({\sf s}) = az_{i+1}z_{i+2}+bz_{i+2}z_{i+1}+cz_i^2,$$ for $i =1,2,3$ with indices taken modulo 3. In this case ${\sf s}^*: Z^{*\otimes 3} \to \kk$ is given by 
{\small \[
{\sf s}_{abc}^*(z^*_{i_1}, z^*_{i_2}, z^*_{i_3}) = 
\begin{cases} 
a & \text{ for } (i_1, i_2, i_3) = (1,2,3), (2,3,1), (3,1,2)\\
b & \text{ for } (i_1, i_2, i_3) = (1,3,2), (3,2,1), (2,1,3)\\
c & \text{ for } (i_1, i_2, i_3) = (1,1,1), (2,2,2), (3,3,3)\\
0 & \text{ otherwise.}
\end{cases}
\]}
 \end{enumerate}
 \end{example}

The following result ensures that a host of well-behaved algebras can be recovered as superpotential algebras associated to preregular forms.

\begin{proposition}\cite[Theorem~11]{DV2007} \label{propA} 
If $A$ is an $N$-Koszul, Artin-Schelter regular algebra of global dimension $d$ generated by $n$ elements in degree one, then $A = A({\sf s}, N)$ for some $\phi$-twisted preregular superpotential ${\sf s}\in (\kk^{n})^{\otimes m}$  with $m \geq N \geq 2$. 
We also get that
$m \equiv 1 \textnormal{ mod } N$ when $N \geq 3$, and $m = d$ when $N=2$.\qed 
\end{proposition}

According to the result above, that ${\sf s}$ is preregular  is a necessary condition for $A({\sf s},N)$ being $N$-Koszul Artin-Schelter regular. Nonetheless, a sufficient condition for a superpotential algebra being $N$-Koszul Artin-Schelter regular remains elusive.

Placing ourselves in the setting of \Cref{propA}, the $N$-Koszul complex of $A({\sf s},N)$ can be recovered from the superpotential ${\sf s}$ as follows. Consider the preliminary result below.

\begin{lemma}[$\partial^{m-N}(\kk {\sf s})^\perp$]\label{L:ANNR}
Take $Z = \bigoplus_{i=1}^n \kk z_i$, an $n$-dimensional $\kk$-vector space. For a superpotential ${\sf s} = \sum_{i_1, \dots, i_m=1}^n ~{\sf s}_{i_1\cdots i_m}z_{i_1}\cdots z_{i_m}\in Z^{\otimes m}$, the annihilator of $\partial^{m-N}(\kk {\sf s})$ in $(Z^*)^{\otimes N}$ is spanned by ${\sf t}=\sum_{i_1,\dots,i_N=1}^n {\sf t}_{i_1\cdots i_N}z_{i_1}^*\cdots z_{i_N}^*$ satisfying 
\begin{align*}
\ts \sum_{i_1,\dots,i_N=1}^n {\sf s}_{\lambda_1\cdots \lambda_{m-N}i_1\cdots i_N}{\sf t}_{i_1\dots i_N}=0,
\end{align*}
for all $\lambda_1,\dots,\lambda_{m-N}\in \{1,\dots,n\}$. We denote this annihilator by $\partial^{m-N}(\kk {\sf s})^\perp\subseteq (Z^*)^{\otimes N}$. \qed
\end{lemma}

Now the result below holds by the proof of \cite[Theorem~11]{DV2007}.

\begin{proposition} \label{N-Koszul} Retain the notation above. For a superpotential 
$$\textstyle {\sf s} = \sum_{i_1, \dots, i_m=1}^n ~{\sf s}_{i_1\cdots i_m}z_{i_1}\cdots z_{i_m},$$ we define a subspace $W^{(r)} \subset Z^{\otimes (m-r)}$ for any $0 \leq r \leq m-N$ to be the span of the  elements  $$\textstyle \sum_{\mu_1, \dots, \mu_{m-r}=1}^n~{\sf s}_{\lambda_1\cdots\lambda_r\mu_1\cdots\mu_{m-r}} z_{\mu_1}\cdots z_{\mu_{m-r}},$$ for all $\lambda_1,\dots,\lambda_r \in \{1, \dots, n\}$.  Take $A$ to be the superpotential algebra $A({\sf s}, N) = TZ/\partial^{m-N}(\kk {\sf s})$, and consider $A^! := TZ^*/\partial^{m-N}(\kk {\sf s})^\perp$, the $N$-Koszul dual of $A$. Then, the $N$-Koszul resolution of the trivial left $A$-module  is given by
\begin{equation}\label{eq:Kosz_N}
  0\to A \otimes \cW_m\overset{d}{\longrightarrow} A \otimes \cW_{m-1}\overset{d^{N-1}}{\longrightarrow} \cdots\overset{d}{\longrightarrow}A \otimes \cW_N\overset{d^{N-1}}{\longrightarrow} A \otimes \cW_1\overset{d}{\longrightarrow}\cW_0 \to \kk \to 0,
\end{equation}
where
\begin{enumerate}
\item $\cW_0 = A$, $\cW_1 = (A_1^!)^* = Z$, and $\cW_\ell = (A_{\ell}^!)^* = W^{(m - \ell)}$ 
for $N \leq \ell \leq m$, and
\item the differentials are induced by maps 
\[d:=\mathrm{mult}\otimes \mathrm{id}_{Z^{\otimes \ell}}: A\otimes Z^{\otimes(\ell+1)}=(A\otimes Z)\otimes Z^{\otimes \ell} \to A\otimes Z^{\otimes \ell}.\]
\end{enumerate}
In particular, we obtain that
\begin{enumerate}
\item[(c)] 
$\cW_m=W^{(0)} = \kk {\sf s}$, so that the last term of \eqref{eq:Kosz_N} is $A \otimes \kk {\sf s}$, 
\item[(d)] $\cW_N = W^{(m-N)}=\partial^{m-N}(\kk {\sf s})$,  the span of elements $\sum_{\mu_1, \dots, \mu_{N}=1}^n~{\sf s}_{\lambda_1\cdots\lambda_{m-N}\mu_1\cdots\mu_{N}} z_{\mu_1}\cdots z_{\mu_{N}}$ for all $\lambda_1, \dots, \lambda_{m-N} \in \{1, \dots, n\}$, and
\item[(e)] the preregular form ${\sf s}^*: (Z^*)^{\otimes m}=(A^!_1)^{\otimes m}\to A^!_m\cong \kk$ given by  multiplication of $A^!$.
\qed
\end{enumerate}
\end{proposition}


 \section{Hopf coaction and homological codeterminant}
We say that a Hopf algebra $H$ {\it coacts} on a vector space $V$ (resp., on an algebra $A$) if $V$ is an $H$-{\it comodule} (resp., if $A$ is an $H$-{\it comodule
algebra}). It is sometimes useful to restrict ourselves to $H$-coactions that do not factor through coactions of `smaller' Hopf algebras. For this, we provide the definition of inner-faithfulness.
\begin{definition}[inner-faithful] \label{innerfaithful}
Let $A$ be a right $H$-comodule (algebra)
with comodule structure map $\rho: A \rightarrow A \otimes H$. We say
that this coaction is {\it inner-faithful} if $\rho(A)\not\subset
A\otimes H'$ for any proper Hopf subalgebra $H'\subsetneq H$.
\end{definition}

Now we recall an important invariant of Hopf algebra coactions on AS-regular algebras: homological codeterminant. To do so, note that a connected graded algebra $A$ is AS-regular if and only
if the Yoneda algebra $E(A):=\bigoplus_{i\geq 0} \Ext^i_A(_A\kk,_A\kk)$ of $A$ is graded Frobenius \cite[Corollary D]{LPWZ:Koszul}. 
\begin{definition}[homological codeterminant ${\sf D}$] \label{hcodet}
Let $A$ be an Artin-Schelter regular algebra with graded Frobenius Yoneda algebra $E(A)$ as above.
Let $H$ be a Hopf algebra with bijective antipode $S$ coacting on
$A$ from the right (resp., from the left). We get that $H$ coacts on $E(A)$
from the left (resp., from the right). Suppose ${\mathfrak e}$ is a nonzero element in (so a $\kk$-basis of)
$\Ext^d_A(_A\kk,_A\kk)$ where $d=\gldim A$.
\begin{enumerate}
\item The {\it homological codeterminant} of the
$H$-coaction on $A$, is defined to be an element ${\sf
D} \in H$  where $\rho({\mathfrak e})={\sf D}\otimes {\mathfrak e}$ (resp.,  $\rho({\mathfrak e})={\mathfrak e} \otimes {\sf D}^{-1}$). Note that  ${\sf D}$ is a grouplike element of $H$. 
\item We say the homological
codeterminant is {\it trivial} if ${\sf D}=1_H$.
\end{enumerate}
\end{definition}

The following proposition is a generalization of \cite[Theorem 2.1]{CKWZ} for computing the homological codeterminant for a Hopf coaction on a (higher-dimensional) $N$-Koszul AS-regular algebra.  
\begin{theorem} \label{thm:codet} Let $m \geq N \geq 2$ be integers.
Let $Z$ be an $n$-dimensional $\kk$-vector space and take a twisted superpotential ${\sf s} \in Z^{\otimes m}$. Let $A:=A({\sf s},N)=TZ/\partial^{m-N}(\kk {\sf s})$ be an $N$-Koszul AS-regular algebra of global dimension $d$. Let $H$ be a Hopf algebra with bijective antipode that coacts on $A$ from the right \textnormal{(}resp., from the left\textnormal{)} via $\rho$, preserving the grading of~$A$, with homological codeterminant ${\sf D}$.
Then, the following statements hold:
\begin{enumerate}
\item  $\kk {\sf s}$ is an one-dimensional $H$-subcomodule of $Z^{\otimes m}$ so that $\rho({\sf s})={\sf s}\otimes {\sf D}^{-1}$ \textnormal{(}resp., $\rho({\sf s})= {\sf D}\otimes {\sf s}$\textnormal{)}. 
\item  {\sf D} is trivial if and only if $\kk{\sf s}$ is the trivial right \textnormal{(}resp., left\textnormal{)} $H$-comodule.
\end{enumerate}
\end{theorem}

\begin{proof}
(a) Giving a right $H$-coaction on $A$ is equivalent to giving a right $H$-comodule structure on $Z$ so that $R:=\partial^{m-N}(\kk {\sf s})$ becomes a right $H$-subcomodule of $Z^{\otimes N}$. Further, if $A^! = TZ^*/R^\perp$ is the $N$-Koszul dual of $A$, then $(A_{N+r}^!)^* = \bigcap_{s+t=r} Z^{\otimes s} \otimes R\otimes Z^{\otimes t}$ is a right $H$-subcomodule of $Z^{\otimes(N+r)}$ for $0 \leq r \leq m-N$. By Proposition \ref{N-Koszul}(a,c), we see that $\kk {\sf s}=(A^!_m)^*$ is a one-dimensional right $H$-subcomodule of $Z^{\otimes m}$.

Now we follow \cite[Remark~1.6(d) and the comments below]{CWZ:Nakayama}  to compute the homological codeterminant of the right $H$-coaction on $A$. Note that the differentials in \cref{eq:Kosz_N} are right $H$-comodule maps, as they are induced by 1-fold or $(N-1)$-fold left multiplication of $A$. The projection map $A \to \kk$ in \cref{eq:Kosz_N} is also a right $H$-comodule map. Hence, the complex \cref{eq:Kosz_N} is indeed a resolution of $\kk$ in the category of graded left-right $(A,H)$-Hopf modules. In order to find the $H$-coaction on $\Ext_A^d(\kk,\kk)$, we apply the functor $\text{Hom}_A(-,\kk)$ to the resolution \cref{eq:Kosz_N} to obtain a complex of $H$-comodules. It is clear that the resulting complex has zero differentials due to the minimality of~\cref{eq:Kosz_N} and simplicity of $\kk$. Therefore, we get left $H$-comodule isomorphisms
\[
\Ext^d_A(\kk, \kk) \cong  \Hom_A(A \otimes \kk{\sf s}, \kk) \cong \Hom_\kk(\kk {\sf s}, \kk) = (\kk {\sf s})^*\subset (Z^*)^{\otimes m};
\]
see \cref{N-Koszul}(c) for the first isomorphism. We consider $(\kk {\sf s})^*$ as a one-dimensional left $H$-subcomodule of $(Z^*)^{\otimes m}$ by applying the antipode of $H$. Thus, the left $H$-coaction on $\Ext^d_A(\kk,\kk)$ is the left $H$-coaction on $(\kk {\sf s})^*$. By \cref{hcodet}, $\rho_E(\mathfrak{e}) = {\sf D} \otimes \mathfrak{e}$ for any basis $\mathfrak{e}$ of $\Ext^d_A(\kk,\kk)$. Therefore, we have $\rho({\sf s})={\sf s} \otimes {\sf D}^{-1}$ by applying the antipode of $H$  and by using the isomorphisms of $H$-comodules above.

Likewise, the parenthetical statement holds.
\smallskip

(b) This follows immediately from part (a).
\end{proof}

 
  \section{Standing hypotheses, notation, and preliminary results}

In this section, we set notation that will be employed throughout this work. We also establish preliminary results on Hopf coactions on $N$-Koszul Artin-Schelter regular algebras (or on superpotential algebras defined by a preregular superpotential by Proposition~\ref{propA}).
  
\begin{notation}\label{not}  Let $m \geq N \geq 2$ be integers.

 \begin{itemize}
\item $V$ is an $n$-dimensional $\kk$-vector space with basis $\{v_1, \dots, v_n\}$.
 \smallskip
 
 \item $V^*$ is the vector space dual of $V$ with dual basis $\{\theta_1,\dots, \theta_n\}$.
 \smallskip
  
 \item $ \e =( \e , \sigma): V^{\otimes m} \to \kk$ is a  preregular $m$-linear form,  which we identify with a $\sigma$-twisted preregular superpotential $\e^*\in (V^{\otimes m})^* \cong (V^*)^{\otimes m}$ by \cref{super-prereg}(b).
 \smallskip
  
 \item  $ \e _{i_1 \cdots i_m}:=  \e (v_{i_1} \otimes \dots \otimes v_{i_m})$.
 \smallskip
  
 \item $ \f=(\f, \psi): (V^*)^{ \otimes m} \to \kk$ is a preregular $m$-linear form, which we identify with a $\psi$-twisted preregular superpotential $\f^*\in ((V^*)^{\otimes m})^* \cong V^{\otimes m}$ by \cref{super-prereg}(b).

 \smallskip
  
 \item $ \f _{i_1 \cdots i_m}:=  \f (\theta_{i_1} \otimes \dots \otimes \theta_{i_m})$.
 \smallskip
   
 \item $(M_\e,\hat\e)$ consists of a one-dimensional $\kk$-vector space $M_\e$ with basis $m_\e$ and a linear map $\hat \e: V^{\otimes m}\to M_\e$ given by $\hat \e (v_{i_1}\otimes \dots \otimes v_{i_m})=\e_{i_1\cdots i_m}m_\e$.   
  \smallskip
  
   \item $(M_\f,\hat \f)$ consists of a one-dimensional $\kk$-vector space $M_\f$ with basis $m_\f$ and a linear map $\hat \f: (V^*)^{\otimes m}\to M_\f$ given by  $\hat \f(\theta_{i_1}\otimes \dots \otimes \theta_{i_m})=\f_{i_1\cdots i_m}m_\f$.   
  \smallskip
   
 \item $A( \e , N) = T V^*/ \partial^{m-N}(\kk  \e^* )$ is the superpotential algebra associated to $ \e $, which by \linebreak \cref{super-prereg}(d) is isomorphic to  
 $$\frac{\kk\langle x_1, \dots, x_n \rangle}{\left(\sum_{j_1, \dots, j_N=1}^n  \e _{{i_1} \cdots {i_{m-N}} {j_1} \cdots {j_N}} x_{j_1}\cdots x_{j_N}\,\big |\, 1\le i_1,\dots,i_{m-N}\le n\right)}.$$ 
Here, the basis $\{\theta_i\}$ of $V^*$ is identified with the indeterminate set $\{x_i\}$.
 \smallskip
   
  \item $A( \f , N) = T V/\partial^{m-N}(\kk  \f^* )$ is the superpotential algebra associated to $ \f $, which by  \linebreak \cref{super-prereg}(d) is isomorphic to  
 $$\frac{\kk\langle y_1, \dots, y_n \rangle}{\left(\sum_{j_1, \dots, j_N=1}^n   \f _{{i_1} \cdots {i_{m-N}} {j_1} \cdots {j_N}}y_{j_1}\cdots y_{j_N}\,\big |\, 1\le i_1,\dots,i_{m-N}\le n\right)}.$$ 
Here, the basis $\{v_i\}$ of $V$ is identified with the indeterminate set $\{y_i\}$.

 \smallskip
   
 \item $\mathbb{P}$ is the matrix representing $\sigma \in GL(V) = GL_n(\kk)$.
 \smallskip
   
  \item $\mathbb{Q}$ is the matrix representing $\psi \in GL(V^*)=GL_n(\kk)$.
 \smallskip
  \item superscripts $T$ and $-T$ of matrices mean transpose and transpose-inverse, respectively.
  \smallskip
\item ${\sf D}_H=:{\sf D}$ is the homological codeterminant of a coaction of a Hopf algebra $H$ on a fixed algebra $A$.
\end{itemize}  
 \end{notation}

Moreover, \cref{propA} prompts the standing assumptions below.
\begin{hypothesis}\label{hyp}
Throughout the paper, we assume that 
\begin{enumerate}
\item superpotential algebras $A( \e , N)$ and $A( \f , N)$ are all $N$-Koszul Artin-Schelter regular; and
\item all Hopf coactions on $N$-Koszul AS-regular algebras (or on superpotential algebras) $A$ preserve the grading of $A$. 
\end{enumerate}
\end{hypothesis}

The non-degeneracy condition in \Cref{preregular}(a) allows preregular forms $\e$ and $\f$ have  inverses in the following sense.       

\begin{lemma}[$\widetilde \e $, $\widetilde \f $] \label{polar}
There exist $m$-linear forms $\widetilde{ \e }: (V^*)^{\otimes m} \to \kk$ and $\widetilde{ \f }: V^{\otimes m} \to \kk$ so that 
\begin{align}
\ts \sum_{k_1, \dots, k_{m-1} =1}^n \widetilde \e _{i k_1 \cdots k_{m-1}}  \e _{k_1 \cdots k_{m-1} j} ~&=~ \delta_{ij}, \label{polar-e}\\
\ts \sum_{k_1, \dots, k_{m-1} =1}^n  \f _{i k_1 \cdots k_{m-1}}\widetilde \f _{k_1 \cdots k_{m-1} j}  ~&=~ \delta_{ij}. \label{polar-f}
\end{align}
\end{lemma}

\begin{proof}
For the $m$-linear preregular form $\e: V^{\otimes m}\to \kk$, we define a map $\varphi: V^{\otimes (m-1)}\to V^*$ such that $\langle \varphi(v_{i_2}\otimes \cdots \otimes v_{i_m}), v_{i_1}\rangle =\e(v_{i_1}\otimes \cdots \otimes v_{i_m})$, where $\langle-,-\rangle: V^*\times V\to \kk$ is the natural evaluation. In view of \Cref{preregular}(a), one sees that the non-degeneracy of $\e$ in the first coordinate is equivalent to the surjectivity of $\varphi$. (For instance,  suppose that $\varphi$ is surjective. Then, \linebreak $\theta_{i_1} = \varphi(\sum_{i_2, \dots, i_m=1}^n \gamma_{i_2 \dots i_m} v_{i_2} \otimes \cdots \otimes v_{i_m})$ for $\gamma_{i_2 \cdots i_m} \in \kk$. 
If $\e$ is degenerate, then for some $v_{i_1} \neq 0$ we have $\e_{i_1 i_2 \cdots i_m}=0$ for any  $i_2, \dots, i_m \in \{1,\dots,n\}$. Thus $\langle \theta_{i_1}, v_{i_1} \rangle = \sum_{i_2, \dots, i_m = 1}^n \gamma_{i_2 \cdots i_m} \e_{i_1 i_2 \cdots i_m} = 0$, which yields a contradiction.) 
Hence we can take a section of $\varphi$, denoted by $\phi: V^*\to V^{\otimes (m-1)}$, such that $\varphi\circ \phi=\mathrm{id}_{V^*}$. We define $\widetilde{\e}: (V^*)^{\otimes m}\to \kk$ by $$\widetilde{\e}(\theta_{i_1}\otimes \cdots \otimes \theta_{i_m})=\langle \theta_{i_1}\otimes \cdots \otimes \theta_{i_{m-1}},\phi(\theta_{i_m})\rangle.$$ Thus by writing $\phi(\theta_{i_{m+1}}) = \sum_{j_2, \dots, j_m=1}^n \alpha_{j_2 \cdots j_m} v_{j_2} \otimes \cdots \otimes v_{j_m}$ for $\alpha_{j_2 \cdots j_m} \in \kk$ we get
\[
\begin{array}{l}
\medskip

\ts \sum_{i_2,\dots,i_m=1}^n \e(v_{i_1}\otimes \cdots \otimes v_{i_m})\, \widetilde{\e}(\theta_{i_2}\otimes \cdots \otimes \theta_{i_{m+1}})\\

\medskip

\quad =\ts \sum_{i_2,\dots,i_m=1}^n \langle \varphi(v_{i_2}\otimes \cdots \otimes v_{i_m}), v_{i_1}\rangle\, \langle \theta_{i_2}\otimes \cdots \otimes \theta_{i_m},\phi(\theta_{i_{m+1}})\rangle \\

\medskip

\quad=  \ts \sum_{i_2,\dots,i_m=1}^n \langle\varphi(v_{i_2}\otimes \cdots \otimes v_{i_m}), v_{i_1}\rangle\, \langle \theta_{i_2}\otimes \cdots \otimes \theta_{i_m},\sum_{j_2,\dots,j_m=1}^n\alpha_{j_2 \cdots j_m} v_{j_2} \otimes \cdots \otimes v_{j_m}\rangle\\

\medskip

\quad= \ts  \left\langle \varphi\left(\sum_{j_2,\dots,j_m=1}^n\alpha_{j_2 \cdots j_m}v_{j_2}\otimes \cdots \otimes v_{j_m}\right), v_{i_1}\right\rangle
~~= \langle \varphi\circ \phi(\theta_{i_{m+1}}),v_{i_1}\rangle ~~=\langle \theta_{i_{m+1}},v_{i_1}\rangle ~~=\delta_{i_1i_{m+1}}.
\end{array}
\]

Clearly, the cyclic condition in \Cref{preregular}(b) implies that any preregular form is also nondegenerate in the last component. Then $\widetilde{\f}$ can be constructed in a similar way.
\end{proof}

\begin{definition}[$\text{Aff}( \e )$, $\text{Aff}( \f )$] 
Let  $\text{Aff}( \e )$ (resp., $\text{Aff}( \f )$) be the set of forms $\widetilde \e $ satisfying \eqref{polar-e} (resp., forms $\widetilde \f $ satisfying \eqref{polar-f}). These are referred to as {\it polar forms}.
\end{definition}

As noted in \cite[page~459]{BDV}, both $\text{Aff}( \e )$ and $\text{Aff}( \f )$ consist of one element if $m=2$, and both consist of more than one element if $m>2$. 

\medbreak On the other hand, as a consequence of \Cref{preregular}(b), we obtain the following notation and result (cf. \cite[(3.2) and (3.3)]{BDV}).

\begin{lemma}[$\mathbb{P}_{k\ell}$, $\mathbb{Q}_{k\ell}$] \label{rotate}
 Let $\mathbb{P}_{k\ell}$ be the components of $\mathbb{P}$ so that 
 $\mathbb{P}(v_\ell) = \sum_{k =1}^n \mathbb{P}_{k\ell}v_k$. Let $\mathbb{Q}_{k\ell}$ be the components of $\mathbb{Q}$ so that 
 $\mathbb{Q}(\theta_\ell) = \sum_{k =1}^n \mathbb{Q}_{k\ell}\theta_k$.
 Then,
\begin{align}
 \ts \sum_{k=1}^n\mathbb{P}_{k\ell} \e _{k i_1 \cdots i_{m-1}} =  \e _{i_1 \cdots i_{m-1} \ell},  \quad  \ts \sum_{\ell=1}^n\mathbb{P}_{\ell k}^{-1} \e _{i_1 \cdots i_{m-1} \ell} =  \e _{k i_1 \cdots i_{m-1}}, \label{rotate1}\\
\ts \sum_{k=1}^n\mathbb{Q}_{k\ell} \f _{k i_1 \cdots i_{m-1}} =  \f _{i_1 \cdots i_{m-1} \ell},\quad \ts \sum_{\ell=1}^n\mathbb{Q}_{\ell k}^{-1}  \f _{i_1 \cdots i_{m-1} \ell} =  \f _{k i_1 \cdots i_{m-1}}. \label{rotate2}
 \end{align}
\qed
 \end{lemma}

We end this section with two results that are crucial to our work.

 \begin{theorem}\label{T:CoactionA(e)}
 Let $K$ be a Hopf algebra coacting on the free algebra $TV^*=\kk \langle x_1,\dots,x_n\rangle$ with $\rho(x_i)=\sum_{j=1}^n a_{ij}\otimes x_j$ for some $a_{ij}\in K$. Consider the $N$-Koszul Artin-Schelter regular algebra $A(\e,N)=TV^*/(R)$, where $R=\partial^{m-N}(\kk \e^*)\subset (V^*)^{\otimes N}$. Then, the following are equivalent.
 
 \smallskip
 
\begin{enumerate}
\item The map $\rho: TV^*\to K\otimes TV^*$ induces naturally a coaction of $K$ on $A(\e,N)$ such that $A(\e,N)$ is a left $K$-comodule algebra. 

\smallskip

\item The coaction $\rho$ satisfies $\rho(R)\subseteq K\otimes R$, that is,
\begin{align*}
\ts \sum_{i_1,\dots,i_N,j_1,\dots,j_N=1}^n \e_{\lambda_1\cdots \lambda_{m-N}i_1\cdots i_N}\mu_{j_1\cdots j_N}\, a_{i_1j_1}\cdots a_{i_Nj_N}=0,
\end{align*} 
where $\ts \sum_{i_1,\dots,i_N=1}^n {\sf e}_{\lambda_1 \cdots \lambda_{m-N} i_1 \cdots i_N} \mu_{i_1 \cdots i_N} = 0$ for all $ \lambda_1,\dots \lambda_{m-N}\in \{1,\dots,n\}$.

\smallskip

\item There is a unique grouplike element $g\in K$ such that the natural inclusion $\kk \e^*\hookrightarrow (V^*)^{\otimes m}$ is a left $K$-comodule map where the coactions of $K$ on $V^*$ and on $\kk \e^*$ are given by $\rho$ and $g \otimes -$, respectively.

\smallskip

\item There is a unique grouplike element $g\in K$ such that the map $\hat\e: V^{\otimes m}\to M_\e$ is a right $K$-comodule map where the coactions of $K$ on $V$ and on $M_\e$  are given by $v_j \mapsto \sum_{i=1}^n v_i\otimes a_{ij}$ and $-\otimes g$, respectively.

\smallskip

\item There is a unique grouplike element $g\in K$ such that the following relations hold in $K$:
\begin{align*}
\ts \sum_{i_1,\dots,i_m=1}^n \e_{i_1\dots i_m}\, a_{i_1j_1}\cdots a_{i_mj_m}=\e_{j_1\cdots j_m}g,
\end{align*} 
for all $j_1,\dots,j_m\in \{1,\dots,n\}$.
\end{enumerate}

\smallskip

\noindent If $K$ satisfies these conditions, then the unique grouplike element $g\in K$ in parts \textnormal{(c)}-\textnormal{(e)} is given by 
 \begin{align*}
 g=\ts \frac{1}{\e_{j_1\dots j_m}}\sum_{i_1,\dots,i_m=1}^n \e_{i_1\dots i_m}\, a_{i_1j_1}\cdots a_{i_mj_m}
 \end{align*}
 for any choice of $j_1,\dots,j_m\in \{1,\dots,n\}$ such that $\e_{j_1\dots j_m}\neq 0$. Moreover, the homological codeterminant of the $K$-coaction on $A(\e,N)$ is given by $g$. 
 \end{theorem}

 \begin{proof}
First, note that the elements $a_{ij}\in K$ given by the $K$-coaction on $TV^*$ span a matrix coalgebra such that $\Delta(a_{ij})=\sum_{k=1}^n a_{ik}\otimes a_{kj}$ and $\varepsilon(a_{ij})=\delta_{ij}$ for all $1\le i,j\le n$. 

\medskip

(a)$\Leftrightarrow$(b) is clear; use Lemma~\ref{L:ANNR} and \cite[Lemma 2.2]{CWZ:Nakayama} for the $m=N=2$ case, for instance. 

\medskip
 
 (d)$\Leftrightarrow$(e) Denote by $\rho(m_\e)=m_\e\otimes g$ the $K$-coaction on $M_\e$ given by the grouplike element $g\in K$. Clearly, $\hat \e: V^{\otimes m}\to M_\e$ is a right $K$-comodule map if and only if the following diagram commutes
\[
\xymatrix{
V^{\otimes m}\ar[d]_-{\rho} \ar[rr]^-{\hat\e} && M_\e\ar[d]^-{-\otimes g}\\
V^{\otimes m}\otimes K \ar[rr]_-{\hat \e \otimes \mathrm{id}_K} && M_\e \otimes K.
}
\]
It is straightforward to check that the commutativity of the diagram above is equivalent to the following condition. 
\begin{align*}
\ts \sum_{i_1,\cdots,i_m=1}^n \e_{i_1\cdots i_m}a_{i_1j_1}\cdots a_{i_mj_m}=e_{j_1\cdots j_m}g,
\end{align*}  
for all $j_1,\dots,j_m\in\{1,\dots,n\}$.

\medskip
 
(c)$\Leftrightarrow$(e) follows similarly. 
 
\medskip
 
(b)$\Rightarrow$(c) Since $R$ is a $K$-subcomodule of $(V^*)^{\otimes N}$,  $\kk \e^*=\bigcap_{r+s=m-N} (V^*)^{\otimes r}\otimes R\otimes (V^*)^{\otimes s}$ is a $K$-subcomodule of $(V^*)^{\otimes m}$ by Proposition \ref{N-Koszul}(c). Thus it is clear that the natural inclusion $\kk \e^*\hookrightarrow (V^*)^{\otimes m}$ is a $K$-comodule map. Since $\kk \e^*$ is one-dimensional, any $K$-coaction on $\kk \e^*$ must be given by some grouplike element in $K$, which in our setting is uniquely determined by $\rho$. 
 
\medskip
 
(c)$\Rightarrow$(b) We endow $V$ with a left $K$-comodule structure defined by $\rho^\vee(v_j)=\sum_{i=1}^n S(a_{ij})\otimes v_i$. Take $V^*$ a left $K$-comodule via $\rho$. Thus in the category of left $K$-comodules, $V$ is the right dual of $V^*$, where the natural evaluation $\ev: V^*\times V\to \kk$ is a $K$-comodule map. Then we consider the following sequence of maps: 
\begin{equation}\label{eq:comp}
    \begin{tikzpicture}[auto,baseline=(current  bounding  box.center)]
    \path[anchor=base] (0,0) node (1) {$\kk \e^* \otimes V^{\otimes(m-N)}$} +(6,0) node (2) {$(V^*)^{\otimes m}\otimes V^{\otimes(m-N)}$}+(13,0) node (3) {$(V^*)^{\otimes N}$};
         \draw[->] (1) to node[pos=.5] {$ \iota\otimes \id_{V^{\otimes(m-N)}}$}  (2);
         \draw[->] (2) to node[pos=.5] {$\id_{(V^*)^{\otimes N}}\otimes \ev_{V^{\otimes(m-N)}}$}  (3);
  \end{tikzpicture}
\end{equation}
where $\iota: \kk \e^*  \to (V^*)^{\otimes m}$ is the inclusion and 
\begin{equation*}
 \ev_{V^{\otimes(m-N)}}: (V^*)^{\otimes(m-N)}\otimes V^{\otimes(m-N)}\to \kk
\end{equation*}
is the iterated evaluation of  $V^*$ against $V$ via the natural evaluation $\ev: V^*\otimes V\to \kk$, with the innermost $V \otimes V$ evaluated first, etc. In view of \cref{super-prereg}(d) (with $r=0,s=m-N$), it is routine to check that 
\[\text{Image}\left(\kk \e^* \otimes V^{\otimes(m-N)}\longrightarrow(V^*)^{\otimes N}\right)=\partial^{m-N}(\kk \e^*)=R.\] 
Now $\iota$ is a map of left $K$-comodules by assumption. Since the evaluation  $\ev: V^*\times V\to \kk$ is also a map of left $K$-comodules, the entire composition \Cref{eq:comp} is preserved by the left $K$-coactions on its domain and codomain. In particular, its image $R$ is a $K$-subcomodule of $(V^*)^{\otimes N}$, or equivalently $\rho(R)\subseteq K\otimes R$ as desired.

\medskip
 
Finally, suppose $K$ satisfies the equivalent conditions (a)-(e). Thus the independency of the choice of $j_1,\dots,j_m\in \{1,\dots,n\}$ in the definition of $g$ is clear from part (e). Moreover, we get $\rho(\e^*)=g \otimes \e^*$ by part (b). Hence the homological codeterminant of the left $K$-coaction on $A(\e,N)$ follows from \cref{thm:codet}(a).
 \end{proof}
 
The following result can be obtained in a similar fashion. 

 \begin{theorem}\label{T:CoactionA(f)}
 Let $K$ be a Hopf algebra coacting on the free algebra $TV=\kk \langle y_1,\dots,y_n\rangle$ with $\rho(y_j)=\sum_{i=1}^n y_i\otimes a_{ij}$ for some $a_{ij}\in K$. Consider the $N$-Koszul Artin-Schelter regular algebra $A(\f,N)=TV/(R)$, where $R=\partial^{m-N}(\kk \f^*)\subset V^{\otimes N}$. Then the following are equivalent.
 
 \smallskip
 
\begin{enumerate}
\item The map $\rho: TV\to TV\otimes K$ induces naturally a coaction of $K$ on $A(\f,N)$ such that $A(\f,N)$ is a right $K$-comodule algebra. 

\smallskip

\item The coaction $\rho$ satisfies $\rho(R)\subseteq R\otimes K$, that is,
\begin{align*}
\ts \sum_{i_1,\dots,i_N,j_1,\dots,j_N=1}^n \f_{\lambda_1\cdots \lambda_{m-N}j_1\cdots j_N}\nu_{i_1\cdots i_N}\, a_{i_1j_1}\cdots a_{i_Nj_N}=0,
\end{align*} 
where $\sum_{j_1, \dots, j_N =1}^n {\sf f}_{\lambda_1, \dots, \lambda_{m-N}, j_1, \dots, j_N}\nu_{j_1\cdots j_N} =0$ for all $ \lambda_1,\dots \lambda_{m-N}\in\{1,\dots,n\}$.

\smallskip

\item There is a unique grouplike element $g\in K$ such that the natural inclusion $\kk \f^*\hookrightarrow V^{\otimes m}$ is a right $K$-comodule map where the coactions of $K$ on $V$ and on $\kk \f^*$  are given by $\rho$ and $-\otimes g$, respectively.

\smallskip

\item There is a unique grouplike element $g\in K$ such that the map $\hat\f: (V^*)^{\otimes m}\to M_\f$ is a left $K$-comodule map where the coactions of $K$ on $V^*$ and on $M_\f$ are given by $\theta_i \mapsto \sum_{j=1}^n a_{ij}\otimes \theta_j$ and $g\otimes -$, respectively.

\smallskip

\item There is a unique grouplike element $g\in K$ such that the following relations hold in $K$:
\begin{align*}
\ts \sum_{j_1,\dots,j_m=1}^n \f_{j_1\dots j_m}\, a_{i_1j_1}\cdots a_{i_mj_m}=\f_{i_1\cdots i_m}g,
\end{align*} 
for all $i_1,\dots,i_m\in \{1,\dots,n\}$.
\end{enumerate}

\smallskip

\noindent If $K$ satisfies these conditions, then the unique grouplike element $g\in K$ in parts \textnormal{(c)}-\textnormal{(e)} is given by 
 \begin{align*}
 g=\ts \frac{1}{\f_{i_1\dots i_m}} \sum_{j_1,\dots,j_m=1}^n \f_{j_1\dots j_m}\, a_{i_1j_1}\cdots a_{i_mj_m}
 \end{align*}
 for any choice of $i_1,\dots,i_m\in \{1,\dots,n\}$ such that $\f_{i_1\dots i_m}\neq 0$. Moreover, the homological codeterminant of the $K$-coaction on $A(\f,N)$ is given by $g^{-1}$. \qed
 \end{theorem}
 
\begin{remark}\label{R:Bialgebra}
Generally, we can consider the coaction of a bialgebra $B$ on $A(\e,N)$ or on $A(\f,N)$, but the adjustments to the theorems above are given as follows. 
\begin{enumerate}
\item Let $B$ coact on the free algebra $TV^*=\kk \langle x_1,\dots,x_n\rangle$ with $\rho(x_i)=\sum_{j=1}^n b_{ij}\otimes x_j$ for some \linebreak $b_{ij}\in B$. Using the same proofs in \cref{T:CoactionA(e)}, we can show that (a)$\Leftrightarrow$(b) and (c)$\Leftrightarrow$(d)$\Leftrightarrow$(e). Moreover, (b) implies (c), but the converse does not necessarily hold.

\smallskip

\item Let $B$ coact on the free algebra $TV=\kk \langle y_1,\dots,y_n\rangle$ with $\rho(y_j)=\sum_{i=1}^n  y_i\otimes b_{ij}$ for some $b_{ij}\in B$. Similarly, we have (a)$\Leftrightarrow$(b) and (c)$\Leftrightarrow$(d)$\Leftrightarrow$(e); again (b)$\Rightarrow$(c), but (c)$\not\Rightarrow$(b).
\end{enumerate}
\end{remark}

\section{The quantum groups $\cH( \e )$, $\cH( \f )$ associated to a preregular form $ \e $, $ \f $, resp.}\label{HeHf}
Recall the notation from the previous section, especially Notation~\ref{not}.
Now we begin the study of universal quantum groups $\cH(\e)$ and $\cH(\f)$ associated to preregular forms $\e$ and $\f$, respectively. We define these Hopf algebras in Sections~\ref{sec:He-pres} and~\ref{sec:Hf-pres}, and show that they coact universally on superpotential algebras $A(\e,N)$ and $A(\f,N)$, respectively, in Section~\ref{sec:He-Hf-prop}.

\subsection{Presentation of $\cH(\e)$}\label{sec:He-pres}

\begin{definition}\label{D:A(e)}
For a preregular form $\e: V^{\otimes m}\to \kk$, let $\cH(\e)$ be the algebra with generators $(a_{ij})_{1\leq i,j \leq n}$, $(b_{ij})_{1\leq i,j \leq n}$, ${\sf D}_ \e ^{\pm 1}$ satisfying the relations 
\begin{align}
\ts \sum_{i_1, \dots, i_m = 1}^n  \e _{i_1 \cdots i_m} a_{i_1 j_1} \cdots a_{i_m j_m}
~&=~  \e _{j_1 \cdots j_m} {\sf D_e},\ \quad \quad \forall\, 1\le j_1,\dots, j_m\le n \label{He-aij},\\
\ts \sum_{i_1, \dots, i_m = 1}^n  \e _{i_1 \cdots i_m} b_{i_m j_m} \cdots b_{i_1 j_1} 
~&=~  \e _{j_1 \cdots j_m} {\sf D_e}^{-1},\  \quad \forall\, 1\le j_1,\dots, j_m\le n \label{He-bij},\\
{\sf D_e}{\sf D_e}^{-1} ~=~ {\sf D_e}^{- 1} {\sf D_e}  ~&=~ 1_{\mathcal{H}} \label{He-Dinv},\\
\A\,\B ~&=~\I, \label{He-antiA}
\end{align}
where $\mathbb A$ and $\mathbb B$ are the matrices $(a_{ij})_{1\leq i,j \leq n}$ and $(b_{ij})_{1\leq i,j \leq n}$, respectively. 
\end{definition}

\begin{lemma}\label{L:RA(e)}
 The following hold in $\cH(\e)$.
\begin{align*}
\mathbb B\, \mathbb A=&\, \I\, =\mathbb A\, \mathbb B,\\
\mathbb B^T\,{\sf D}_\e^{-1}\, \mathbb P^T\, \mathbb A^T\, \mathbb P^{-T}\, {\sf D}_\e=&\, \I\,={\sf D}_\e^{-1}\, \mathbb P^T\, \mathbb A^T\, \mathbb P^{-T}\, {\sf D}_\e\, \mathbb B^T.
\end{align*}
\end{lemma}

\begin{proof}
We first show that $\mathbb B\mathbb A=\I=\mathbb A\mathbb B$. By \cref{polar}, we have an element $\widetilde \e  \in \text{Aff}( \e )$. Let $\mathbb C:=(c_{ij})_{1\le i,j\le n}$ be elements in $\cH(\e)$ such that 
\begin{equation}\label{leftantiA}
c_{ij}=\ts \sum_{i_1, \dots, i_{m-1}, j_1, \dots, j_{m-1}=1}^n ~~{\sf D_e}^{-1} \widetilde{ \e }_{i i_1 \cdots i_{m-1}}a_{j_1 i_1} \cdots a_{j_{m-1} i_{m-1}}  \e _{j_1 \cdots j_{m-1} j}.
\end{equation}
Thus we have
\[
\begin{array}{rl}
\sum_{k=1}^n c_{ik}a_{kj} 
&\overset{\cref{leftantiA}}{=}  \ts \sum_{i_1, \dots, i_{m-1}, k_1, \dots, k_{m-1}, k=1}^n 
~~{\sf D_e}^{-1} \widetilde{ \e }_{i i_1 \cdots i_{m-1}}a_{k_1 i_1} \cdots a_{k_{m-1} i_{m-1}}  \e _{k_1 \cdots k_{m-1} k}a_{kj}\\
&\overset{\cref{He-aij}}{=} \ts \sum_{i_1, \dots, i_{m-1}=1}^n 
~~{\sf D_e}^{-1} \widetilde{ \e }_{i i_1 \cdots i_{m-1}} \e _{i_1 \cdots i_{m-1} j} {\sf D}_ \e  \\
&\overset{\cref{polar-e}}{=}~~ \delta_{ij}.
\end{array}
\]
This shows that $\mathbb C\mathbb A=\I$. By \eqref{He-antiA}, $\mathbb C=\mathbb C\I=\mathbb C(\mathbb A\mathbb B)=(\mathbb C\mathbb A)\mathbb B=\I\mathbb B=\mathbb B$. Thus $\mathbb B\mathbb A=\I=\mathbb A\mathbb B$. 

Next, consider the following computation. 
\[
\begin{array}{l}
({\sf D}_\e^{-1}\, \mathbb P^T\, \mathbb A^T\, \mathbb P^{-T}\, {\sf D}_\e\, \mathbb B^T)_{ij}~=~\sum_{s,t,k = 1}^n ~{\sf D_e}^{-1}\, \mathbb{P}_{ti}\, a_{st}\, \mathbb{P}^{-1}_{ks}\, {\sf D_e}\, b_{jk}~=~\sum_{s,t,k = 1}^n ~{\sf D_e}^{-1}\, \mathbb{P}_{ti}\, a_{st}\, \mathbb{P}^{-1}_{ks}\, {\sf D_e}\, c_{jk}\\
\smallskip
\quad \overset{\cref{leftantiA}}{=} \sum_{s,t,k,j_1,\dots,j_{m-1},k_1,\dots,k_{m-1}= 1}^n ~{\sf D_e}^{-1}\mathbb{P}_{ti}a_{st} \mathbb{P}^{-1}_{ks}  {\sf D_e} ({\sf D_e}^{-1} \widetilde{ \e }_{j j_1 \cdots j_{m-1}} a_{k_1 j_1} \cdots a_{k_{m-1} j_{m-1}}  \e _{k_1 \cdots k_{m-1} k})\\
\smallskip
\quad \overset{\eqref{rotate1}}{=} \sum_{s,t,j_1,\dots,j_{m-1},k_1,\dots,k_{m-1}= 1}^n ~{\sf D_e}^{-1}\mathbb{P}_{ti} ( \e _{s k_1 \cdots k_{m-1}}  a_{st}  a_{k_1 j_1} \cdots a_{k_{m-1} j_{m-1}}) \widetilde{ \e }_{j j_1 \cdots j_{m-1}}\\
\smallskip
\quad \overset{\cref{He-aij}}{=} \sum_{s,j_1,\dots,j_{m-1}= 1}^n ~{\sf D_e}^{-1}\mathbb{P}_{ti}  \e _{t j_1 \cdots j_{m-1}} {\sf D_e} \widetilde{ \e }_{j j_1 \cdots j_{m-1}}\\
\smallskip
\quad \overset{\eqref{rotate1}}{=}  \sum_{j_1,\dots,j_{m-1}= 1}^n ~ \widetilde{ \e }_{j j_1 \cdots j_{m-1}}  \e _{j_1 \cdots j_{m-1}i}\\ 
\quad \overset{\cref{polar-e}}{=} \delta_{ij}.
\end{array}
\]
Then ${\sf D}_\e^{-1}\, \mathbb P^T\, \mathbb A^T\, \mathbb P^{-T}\, {\sf D}_\e\, \mathbb B^T=\I$. Similarly let $\mathbb D:=(d_{ij})_{1\le i,j\le n}$ be elements in $\cH(\e)$ such that 
\begin{equation*}
d_{ij}=\ts \sum_{i_1, \dots, i_{m-1}, j_1, \dots, j_{m-1}=1}^n \widetilde{ \e }_{i i_1 \cdots i_{m-1}}b_{j_{m-1} i_{m-1}} \cdots b_{j_1i_1}  \e _{j_1 \cdots j_{m-1} j}\, {\sf D_e}.
\end{equation*}
It is straightforward to check that $\B^T \mathbb D^T=\I$. Hence, we have ${\sf D}_\e^{-1}\, \mathbb P^T\, \mathbb A^T\, \mathbb P^{-T}\, {\sf D}_\e=\mathbb D^T$ and $\mathbb B^T\,({\sf D}_\e^{-1}\, \mathbb P^T\, \mathbb A^T\, \mathbb P^{-T}\, {\sf D}_\e)=\I$. This completes our proof. 
\end{proof}

\begin{proposition}\label{P:H(e)}
The algebra $\cH(\e)$ admits a Hopf algebra structure, with comultiplication $\Delta$ defined by  
\begin{align}
\quad \Delta(a_{ij}) = \ts \sum_{k=1}^n a_{ik}\otimes  a_{kj}, \quad \Delta(b_{ij}) = \ts \sum_{k=1}^n b_{kj}\otimes  b_{ik}, \quad \Delta({\sf D}_{\sf e}^{\pm 1})={\sf D}^{\pm 1}_{\sf e}\otimes {\sf D}^{\pm 1}_{\sf e},\label{He-coprod}
\end{align}
with counit $\varepsilon$ defined by 
\begin{align}
\varepsilon(a_{ij})=\varepsilon(b_{ij})=\delta_{ij}, \quad  \varepsilon({\sf D}^{\pm 1}_{\sf e})=1 \label{He-counit}
\end{align}
and with antipode $S$ defined by
\begin{align}
S(\mathbb A)=\mathbb B,\quad  S(\mathbb B)={\sf D}_ \e ^{-1}\, \mathbb P^{-1}\, \mathbb A\, \mathbb P\, {\sf D}_ \e,\quad S({\sf D}_{\sf e}^{\pm 1})={\sf D}_{\sf e}^{\mp 1}. \label{He-S}
\end{align}
\end{proposition}
\begin{proof}
One sees easily that \eqref{He-coprod} and \eqref{He-counit} equip the algebra $\cH(\e)$ with a bialgebra structure. It remains to show that $S$ defined in \eqref{He-S} is an antipode of $\cH(\e)$. As a consequence of \cref{L:RA(e)}, we have 
\begin{align}\label{antiHe}
\ts \sum_{k=1}^n a_{ik}S(a_{kj})=\sum_{k=1}^n S(a_{ik})a_{kj}=\varepsilon(a_{ij}), \quad \ts \sum_{k=1}^n b_{kj}S(b_{ik})=\sum_{k=1}^n S(b_{kj})b_{ik}=\varepsilon(b_{ij}), 
\end{align}
for all $1\le i,j\le n$. Thus it suffices to show that $S$ is compatible with the relations in $\cH(\e)$. It is clear that $S$ is compatible with the relation \cref{He-Dinv}. For \cref{He-aij}, \cref{He-bij}, \cref{He-antiA}, consider the following computations:
\[
\begin{array}{rl}
\smallskip
S(\sum_{i_1, \dots, i_m = 1}^n  \e _{i_1 \cdots i_m} a_{i_1 j_1} \cdots a_{i_m j_m}) 
&~= \sum_{i_1, \dots, i_m = 1}^n  \e _{i_1 \cdots i_m} b_{i_m j_m} \cdots b_{i_1 j_1}\\
\smallskip
&\overset{\cref{He-bij}}{=}  \e _{j_1 \cdots j_m} {\sf D_e}^{-1}
\quad ~= S( \e _{j_1 \cdots j_m} {\sf D_e})\\\\

\smallskip
S(\sum_{i_1, \dots, i_m = 1}^n  \e _{i_1 \cdots i_m} b_{i_m j_m} \cdots b_{i_1 j_1}) 
&\overset{\cref{He-bij}}{=} \sum_{\substack{i_1, \dots, i_m, \quad \\k_1, \dots, k_m = 1}}^n S(b_{i_1 j_1}) \cdots S(b_{i_m j_m})b_{k_m i_m} \cdots b_{k_1 i_1} \e _{k_1 \cdots k_m}{\sf D}_ \e \\
\smallskip
&\hspace{-.03in} \overset{\cref{antiHe}}{=} \sum_{k_1, \dots, k_m = 1}^n  \delta_{k_1,j_1} \cdots \delta_{k_m, j_m} \e _{k_1 \cdots k_m}{\sf D}_ \e \\
\smallskip
&~=  \e _{j_1 \cdots j_m} {\sf D_e}
\quad = S( \e _{j_1 \cdots j_m} {\sf D_e}^{-1})\\\\
S(\ts\sum_{k=1}^n a_{ik} b_{kj}) ~&~=\sum_{k=1}^n S(b_{kj}) S(a_{ik})
~=~\sum_{k=1}^n S(b_{kj}) b_{ik} ~\overset{\cref{antiHe}}{=}~ \delta_{ij}.
\end{array}
\]
This gives our result. 
\end{proof}

\begin{remark} We emphasize that the generators $b_{ij}$'s  above are needed to define the ``right inverse" (antipode) of $a_{ij}$'s for $\cH({\sf e})$. Indeed, the nondegeneracy of the preregular form $\e$ (\cref{preregular}(a)) implies that $\mathbb A$ must have a left inverse in terms of $a_{ij}$'s, which in the proof of \cref{L:RA(e)} is denoted by $\mathbb C=(c_{ij})_{1\le i,j\le n}$. Yet we do not necessarily have that $\mathbb A\mathbb C=\I$. If we impose that $\mathbb A\mathbb C=\I$, we obtain a Hopf quotient $\cH( \e , \widetilde \e )$ of $\cH( \e )$  as defined in \cite[Section~8]{DV2007}, where the antipode of  $\cH( \e , \widetilde \e )$ is given by $S(\mathbb A)=\mathbb C$ above.
\end{remark}

We denote by $\eta_{\sf D_{\sf e}}$ and $\eta_{\sf e}$ two automorphisms of $\cH(\sf e)$ satisfying 
\begin{align*}
\eta_{\sf D_{\sf e}}(\mathbb A)&={\sf D_{\sf e}}\mathbb A{\sf D_{\sf e}}^{-1},\hspace{0.2in}\eta_{\sf D_{\sf e}}(\mathbb B)={\sf D_{\sf e}}\mathbb B{\sf D_{\sf e}}^{-1},\hspace{0.2in}\eta_{\sf D_{\sf e}}({\sf D}_{\sf e}^{\pm 1})={\sf D}_{\sf e}^{\pm 1};\\
\eta_{\e}(\mathbb A)&=\mathbb P^{-1}\mathbb A\mathbb P, \hspace{0.42in}\eta_{\e}(\mathbb B)=\mathbb P^{-1}\mathbb B \mathbb P,\hspace{0.44in}\eta_\e({\sf D}_{\sf e}^{\pm 1})={\sf D}_{\sf e}^{\pm 1}.
\end{align*}
By \cref{P:H(e)}, it is clear that $\eta_{\sf D_{\sf e}}$ and $\eta_{\sf e}$ are Hopf algebra automorphisms of $\cH(\sf e)$.

\begin{corollary}\label{C:S2He}
We have $\eta_{\sf D_{\sf e}}\circ S^2=\eta_{\sf e}$ in $\cH({\sf e})$. Thus, the antipode $S$ of $\cH(\e)$ is bijective. 
\end{corollary}
\begin{proof}
It suffices to check for the generators of $\cH(\e)$. We apply \eqref{He-S} to get 
\begin{align*}
\eta_{\sf D_{\sf e}}\circ S^2({\sf D}_{\sf e})&={\sf D}_{\sf e}=\eta_{\sf e}({\sf D}_{\sf e}),\\
\eta_{\sf D_{\sf e}}\circ S^2(\mathbb A)&={\sf D}_{\sf e} S(\mathbb B){\sf D}_{\sf e}^{-1}=\mathbb P^{-1}\mathbb A\mathbb P=\eta_{\sf e}(\mathbb A),\\
\eta_{\sf D_{\sf e}}\circ S^2(\mathbb B)&={\sf D}_{\sf e} S({\sf D}_{\sf e}^{-1}\mathbb P^{-1}\mathbb A\mathbb P{\sf D}_{\sf e}){\sf D}_{\sf e}^{-1}=\mathbb P^{-1}S(\mathbb A)\mathbb P=\mathbb P^{-1}\mathbb B\mathbb P=\eta_{\sf e}(\mathbb B).
\end{align*}
The bijectivity of $S^2$, and thus $S$, follows immediately. 
\end{proof}

Now we establish universal properties of $\cH(\e)$.

\begin{proposition}\label{P:UH(e)} 
\
\begin{enumerate}
\item Endow both $V$ and $M_\e$ with the right $\cH(\e)$-comodule structures defined  respectively by  \linebreak $\rho_\cH(v_j)=\sum_{i=1}^n v_i\otimes a_{ij}$ and $\rho_\cH(m_\e)=m_\e\otimes {\sf D}_\e$. Then, the linear map $\hat \e: V^{\otimes m}\to M_\e$ is an $\cH(\e)$-comodule map. 

\smallskip

\item Let $K$ be a Hopf algebra that right coacts on both $V$ and $M_\e$ so that $\hat \e$ is a $K$-comodule map. Then, there exists a unique Hopf algebra map $\phi: \cH(\e)\to K$ such that $(\mathrm{id}_V\otimes \phi)\circ \rho_\cH=\rho_K$, where $\rho_\cH$ and $\rho_K$ denote the coactions on $V$ of $\cH(\e)$ and $K$, respectively. 

\smallskip

\item Endow $V^*$ with the left $\cH(\e)$-comodule structure defined by $\rho_\cH(\theta_i)=\sum_{j=1}^n a_{ij}\otimes \theta_j$. Then, $\kk \e^*$ is an $\cH(\e)$-subcomodule of $(V^*)^{\otimes m}$. 

\smallskip

\item Let $K$ be a Hopf algebra that left coacts on $V^*$ so that $\kk \e^*$ is a $K$-subcomodule of $(V^*)^{\otimes m}$. Then, there exists a unique Hopf algebra map $\phi: \cH(\e)\to K$ such that $(\phi\otimes \mathrm{id}_{V^*})\circ \rho_\cH=\rho_K$, where $\rho_\cH$ and $\rho_K$ denote the coactions on $V^*$ of $\cH(\e)$ and $K$, respectively. 
\end{enumerate}
\end{proposition}

\begin{proof}
We only prove for part (a,b), and part (c,d) can be obtained in a similar way. 

\smallskip

(a) This follows by \eqref{He-aij} and Theorem~\ref{T:CoactionA(e)}((e)$\Rightarrow$(d)).

\smallskip

(b) Write $\rho_K(v_j)=\sum_{i=1}^n v_i\otimes k_{ij}$ for some $k_{ij}\in K$ and $\rho_K(m_\e)=m_\e\otimes g$ for some grouplike element $g\in K$. By \cref{T:CoactionA(e)}((d)$\Rightarrow$(e)), the following hold in $K$.
\begin{align}\label{E:UPH(e)}
\ts \sum_{i_1,\cdots,i_m=1}^n \e_{i_1\cdots i_m}k_{i_1j_1}\cdots k_{i_mj_m}=e_{j_1\cdots j_m}g,\  \text{for all}\ 1\le j_1,\dots,j_m \le n.
\end{align}  
Furthermore, write $h_{ij}=S(k_{ij})\in K$. Thus, for any $1\le j_1,\dots,j_m\le n$,
\begin{align*}
\ts \sum_{i_1,\dots,i_m=1}^n \e_{i_1\cdots i_m}h_{i_mj_m}\cdots h_{i_1j_1}&\overset{\eqref{E:UPH(e)}}{=}\ts \sum_{i_1,\dots,i_m=1}^n(\sum_{r_1,\dots,r_m=1}^n g^{-1}\e_{r_1\cdots r_m} k_{r_1i_1}\cdots k_{r_mi_m}) h_{i_mj_m}\cdots h_{i_1j_1}\\
&\hspace{.08in}=\ts \sum_{r_1,\dots,r_m=1}^n g^{-1}\e_{r_1\cdots r_m} (\sum_{i_1,\dots,i_m=1}^n k_{r_1i_1}\cdots (k_{r_mi_m}h_{i_mj_m})\cdots h_{i_1j_1})\\
&\hspace{.08in}=\ts \sum_{r_1,\dots,r_m=1}^n \delta_{r_1j_1}\cdots \delta_{r_mj_m} \e_{r_1\cdots r_m}g^{-1}\\
&\hspace{.08in}=\e_{j_1\cdots j_m}g^{-1}.
\end{align*}
Hence, there is a Hopf algebra map $\phi: \cH(\e)\to K$ given by $\phi(a_{ij})=k_{ij}, ~~\phi(b_{ij})=h_{ij}$ and $\phi({\sf D}_\e^{\pm 1})=g^{\pm 1}$. It is clear that $\phi$ is the unique map satisfying $(\mathrm{id}_V\otimes \phi)\circ \rho_\cH=\rho_K$. This proves our result.  
\end{proof}

\subsection{Presentation of $\cH(\f)$}\label{sec:Hf-pres}
For the preregular form $\f: (V^*)^{\otimes m}\to \kk$, we can consider $\cH({\sf f})$ as a dual version of $\cH({\sf e})$ having a similar presentation and universal property. Therefore, the following results for $\cH(\f)$ can be obtained by the same fashion as in the previous section.

\begin{definition}\label{D:A(f)}
For a preregular form $\f: (V^*)^{\otimes m}\to \kk$, let $\cH(\f)$ be the algebra with generators $(a_{ij})_{1\leq i,j \leq n}$, $(b_{ij})_{1\leq i,j \leq n}$, ${\sf D}_ \f ^{\pm 1}$ satisfying the relations 
\begin{align}
\ts \sum_{i_1, \dots, i_m = 1}^n  \f _{i_1 \cdots i_m} a_{j_1 i_1} \cdots a_{j_m i_m} 
~&=~  \f _{j_1 \cdots j_m} {\sf D_f}^{-1},\quad  \forall\, 1\le j_1,\dots, j_m\le n,\label{Hf-aij}\\
\ts \sum_{i_1, \dots, i_m = 1}^n  \f _{i_1 \cdots i_m} b_{j_m i_m} \cdots b_{j_1 i_1} 
~&=~  \f _{j_1 \cdots j_m} {\sf D_f},\quad \quad  \forall\, 1\le j_1,\dots, j_m\le n, \label{Hf-bij}\\
{\sf D_f}{\sf D_f}^{-1} ~=~ {\sf D_f}^{- 1} {\sf D_f}  ~&=~ 1_{\mathcal{H}}, \label{Hf-Dinv}\\
\B\,\A ~&=~\I, \label{Hf-antiA}
\end{align}
where $\mathbb A$ and $\mathbb B$ are the matrices $(a_{ij})_{1\leq i,j \leq n}$ and $(b_{ij})_{1\leq i,j \leq n}$, respectively. 
\end{definition}

\begin{lemma}\label{L:RA(f)}
The following hold in $\cH(\f)$.
\begin{align*}
\mathbb B\, \mathbb A=&\, \I\, =\mathbb A\, \mathbb B,\\
\mathbb B^T\,{\sf D}_\f^{-1}\, \mathbb Q\, \mathbb A^T\, \mathbb Q^{-1}\, {\sf D}_\f=&\, \I\,={\sf D}_\f^{-1}\, \mathbb Q\, \mathbb A^T\, \mathbb Q^{-1}\, {\sf D}_\f\, \mathbb B^T.
\end{align*} 

\vspace{-.25in}

\qed
\end{lemma}

\begin{proposition}\label{P:H(f)}
The algebra $\cH(\f)$ admits a Hopf algebra structure, with comultiplication $\Delta$ defined by  
\begin{align}
\quad \Delta(a_{ij}) = \ts \sum_{k=1}^n a_{ik}\otimes  a_{kj},\quad \Delta(b_{ij}) = \ts \sum_{k=1}^n b_{kj}\otimes  b_{ik}, \quad \Delta({\sf D}_{\sf f}^{\pm 1})={\sf D}^{\pm 1}_{\sf f}\otimes {\sf D}^{\pm 1}_{\sf f},\label{Hf-coprod}
\end{align}
with counit $\varepsilon$ defined by 
\begin{align}
\varepsilon(a_{ij})=\varepsilon(b_{ij})=\delta_{ij}, \quad \varepsilon({\sf D}^{\pm 1}_{\sf f})=1 \label{Hf-counit}
\end{align}
and with antipode $S$ defined by
\begin{align}
S(\mathbb A)=\mathbb B,\quad  S(\mathbb B)={\sf D}_ \f ^{-1}\, \mathbb Q^{-T}\, \mathbb A\, \mathbb Q^T\, {\sf D}_ \f,\quad S({\sf D}_{\sf f}^{\pm 1})={\sf D}_{\sf f}^{\mp 1}. \label{Hf-S}
\end{align}\qed
\end{proposition}

We denote by $\eta_{\sf D_{\sf f}}$ and $\eta_{\sf f}$ two Hopf algebra automorphisms of $\cH(\sf f)$ satisfying 
\begin{align*}
\eta_{\sf D_{\sf f}}(\mathbb A)&={\sf D_{\sf f}}\mathbb A{\sf D_{\sf f}}^{-1},\hspace{0.35in}\eta_{\sf D_{\sf f}}(\mathbb B)={\sf D_{\sf f}}\mathbb B{\sf D_{\sf f}}^{-1},\hspace{0.45in}\eta_{\sf D_{\sf f}}({\sf D}_{\sf f}^{\pm 1})={\sf D}_{\sf f}^{\pm 1};\\
\eta_{\f}(\mathbb A)&=\mathbb Q^{-T}\mathbb A\mathbb Q^T, \hspace{0.4in}\eta_{\f}(\mathbb B)=\mathbb Q^{-T}\mathbb B \mathbb Q^T,\hspace{0.5in}\eta_\f({\sf D}_{\sf f}^{\pm 1})={\sf D}_{\sf f}^{\pm 1}.
\end{align*}

\begin{corollary}\label{C:S2Hf}
We have $\eta_{\sf D_{\sf f}}\circ S^2=\eta_{\sf f}$ in $\cH({\sf f})$. Thus, the antipode $S$ of $\cH(\f)$ is bijective. \qed
\end{corollary}

\begin{proposition}\label{P:UH(f)} 
\
\begin{enumerate}
\item Endow both $V^*$ and $M_\f$ with the left $\cH(\f)$-comodule structures defined respectively by \linebreak $\rho_\cH(\theta_i)=\sum_{j=1}^n a_{ij}\otimes \theta_j$ and $\rho_\cH(m_\f)={\sf D}_\f^{-1}\otimes m_\f$. Then, the linear map $\hat \f: (V^*)^{\otimes m}\to M_\f$ is an $\cH(\f)$-comodule map. 

\smallskip

\item Let $K$ be a Hopf algebra that left coacts on both $V^*$ and $M_\f$ so that $\hat \f$ is a $K$-comodule map. Then, there exists a unique Hopf algebra map $\phi: \cH(\f)\to K$ such that $(\phi\otimes \mathrm{id}_{V^*})\circ \rho_\cH=\rho_K$, where $\rho_\cH$ and $\rho_K$ denote the coactions on $V^*$ of $\cH(\f)$ and $K$, respectively. 

\smallskip

\item Endow $V$ with the right $\cH(\f)$-comodule structure defined by $\rho_\cH(v_j)=\sum_{i=1}^n v_i\otimes a_{ij}$. Then, $\kk \f^*$ is an $\cH(\f)$-subcomodule of $V^{\otimes m}$. 

\smallskip

\item Let $K$ be a Hopf algebra that right coacts on $V$ so that $\kk \f^*$ is a $K$-subcomodule of $V^{\otimes m}$. Then, there exists a unique Hopf algebra map $\phi: \cH(\f)\to K$ such that $(\mathrm{id}_{V}\otimes \phi)\circ \rho_\cH=\rho_K$, where $\rho_\cH$ and $\rho_K$ denote the coactions on $V$ of $\cH(\f)$ and $K$, respectively. \qed
\end{enumerate}
\end{proposition}

\subsection{Coaction on superpotential algebra $A(\e,N)$ or $A(\f,N)$} \label{sec:He-Hf-prop}
In this part, we discuss further universal properties of $\cH( \e )$ and $\cH( \f )$ pertaining to their coactions on superpotential algebras $A( \e ,N)$ and $A( \f ,N)$, respectively.  Recall by Lemma \ref{super-prereg}, we denote by $\e^*\in (V^*)^{\otimes m}$ and $\f^*\in V^{\otimes m}$ the twisted preregular superpotentials corresponding to $\e$ and $\f$. 

First, we define certain quotient Hopf algebras of $\cH(\e)$ and $\cH(\f)$, respectively. 

\begin{definition}\label{S(e) S(f)} [$\cS( \e )$, $\cS( \f )$]
 Let $\cS( \e )$ be the Hopf algebra $\cH(\e)/({\sf D}_\e-1)$ and let $\cS( \f )$ be the Hopf algebra $\cH(\f)/({\sf D}_\f-1)$.
\end{definition}

\begin{remark}\label{R:USL}
In $\cS(\e)$, we will still keep the notation $(a_{ij})_{1\le i,j\le n}$ and $(b_{ij})_{1\le i,j\le n}$ for their images under the quotient map $\cH(\e)\twoheadrightarrow \cS(\e)$. Then $\cS(\e)$ has generators $(a_{ij})_{1\le i,j\le n}$ and $(b_{ij})_{1\le i,j\le n}$ subject to relations \eqref{He-aij}, \eqref{He-bij}, \eqref{He-antiA} with ${\sf D}_\e^{\pm 1}=1$. Moreover, $\cS(\e)$ has universal properties like those of $\cH(\e)$ in \Cref{P:UH(e)} by requiring the corresponding coactions on $M_\e$ or $\kk \e^*$ to be trivial. 

Similar statements hold for $\cS(\f)$ as well. 
\end{remark}

Now we discuss inner-faithful coactions on superpotential algebras.

\begin{proposition}\label{HeHf-coacts} We have the following inner-faithful coactions:
\begin{enumerate}
\item $\cH( \e )$ left coacts on $A( \e ,N)$ inner-faithfully via $\rho_\cH(x_i) = \textstyle \sum_{j=1}^n a_{ij} \otimes x_j, 1\le i\le n$ with homological codeterminant ${\sf D}_\e$.

\smallskip

\item $\cS( \e )$ left coacts on $A( \e ,N)$ inner-faithfully via $\rho_\cS(x_i) = \textstyle \sum_{j=1}^n a_{ij} \otimes x_j, 1\le i\le n$ with trivial homological codeterminant.

\smallskip

\item  $\cH( \f )$ right coacts on $A( \f ,N)$ inner-faithfully via $\rho_\cH(y_j) = \textstyle \sum_{i=1}^n y_i \otimes a_{ij}, 1\le j\le n$ with homological codeterminant  ${\sf D}_\f$.

\smallskip

\item  $\cS( \f )$ right coacts on $A( \f ,N)$ inner-faithfully via $\rho_\cS(y_j) = \textstyle \sum_{i=1}^n y_i \otimes a_{ij}, 1\le j\le n$ with trivial homological codeterminant,
\end{enumerate}
\end{proposition}

\begin{proof}
Part (a) follows from \cref{T:CoactionA(e)}((e)$\Rightarrow$(a)) and relation \eqref{He-aij}; the homological codeterminant is also given by \cref{T:CoactionA(e)}. Suppose $K$ is any Hopf subalgebra of $\cH(\e)$ through which the coaction of $\cH(\e)$ factors. Then one sees that $a_{ij}\in K$. Moreover, $b_{ij}=S(a_{ij})\in K$ and ${\sf D}_\e^{\pm 1}\in K$ by \cref{T:CoactionA(e)}((a)$\Rightarrow$(e)). So $K=\cH(\e)$, which proves the inner-faithfulness of the coaction of $\cH(\e)$. The remaining parts follow in the same way. 
\end{proof}

Now we consider a version of Manin's universal quantum groups that coact on the $N$-Koszul Artin-Schelter regular algebras $A( \e ,N)$ and $A( \f ,N)$; we refer the reader to \cite[Section~2.2]{CWZ:Nakayama} and \cite[Section~7.5]{Manin:QGNCG} for background material.

\begin{definition}[one-sided quantum general/ special linear groups $\cO_{A(*, N)}(GL)$/ $\cO_{A(*, N)}(SL)$] \label{OA(e) OA(f)} 
\
\begin{enumerate}
\item The {\it \textnormal{(}left\textnormal{)} quantum general linear group of $A( \e ,N)$}, denoted by $\cO_{A( \e ,N)}(GL)$, is defined to be the Hopf algebra that left coacts on $A( \e ,N)$ universally so that for any Hopf algebra $K$ that left coacts on $A( \e ,N)$, there is a unique Hopf algebra map $\phi: \cO_{A( \e ,N)}(GL)\to K$ such that $(\phi\otimes \mathrm{id}_{A(\e,N)})\circ \rho_\cO=\rho_K$, where $\rho_\cO$ and $\rho_K$ denote the coactions on $A(\e,N)$ of $\cO_{A( \e ,N)}(GL)$ and $K$, respectively.

\medskip

\item The {\it \textnormal{(}left\textnormal{)} quantum special linear group of $A(\e,N)$}, denoted by $\cO_{A(\e, N)}(SL)$, is defined to be the Hopf algebra that left coacts on $A(\e, N)$ universally with trivial homological codeterminant so that for any Hopf algebra $K$ that left coacts on $A(\e, N)$ with trivial homological codeterminant, there is a unique Hopf algebra map $\phi: \cO_{A(\f ,N)}(SL)\to K$ such that $(\phi\otimes \mathrm{id}_{A(\e,N)})\circ \rho_\cO=\rho_K$, where $\rho_\cO$ and $\rho_K$ denote the coactions on $A(\e, N)$ of $\cO_{A(\e, N)}(SL)$ and $K$, respectively.

\medskip

\item The {\it \textnormal{(}right\textnormal{)} quantum general linear group of $A(\f, N)$}, denoted by $\cO_{A(\f, N)}(GL)$,  can be similarly defined as in part (a) by changing the left coaction into the right coaction.  

\medskip

\item The {\it \textnormal{(}right\textnormal{)} quantum special linear group of $A(\f, N)$}, denoted by $\cO_{A(\f, N)}(SL)$,  can be similarly defined as in part (b) by changing the left coaction into the right coaction.  
\end{enumerate}
\end{definition}

\begin{theorem}\label{pr.univ}
\
\begin{enumerate}
\item
The left coaction of $\cH(\e)$ \textnormal{(}resp., of $\cS(\e)$\textnormal{)} on $A(\e,N)$ given in \cref{HeHf-coacts}(a) \textnormal{(}resp.,~(b)\textnormal{)} satisfy the universal property described by \cref{OA(e) OA(f)}(a) \textnormal{(}resp.,~(b)\textnormal{)}. Therefore, we have the following equalities of universal quantum groups
\begin{align*}
\cH( \e ) = \cO_{A( \e ,N)}(GL)\quad \text{ and } \quad 
\cS(\e ) =  \cO_{A( \e ,N)}(SL).
\end{align*}

\item The right coactions of $\cH(\f)$ \textnormal{(}resp., of $\cS(\f)$\textnormal{)} on $A(\f,N)$ given in \cref{HeHf-coacts}(c) \textnormal{(}resp.,~(d)\textnormal{)} satisfy the universal properties described by \cref{OA(e) OA(f)}(c) \textnormal{(}resp.,~(d)\textnormal{)}. Therefore, we have the following equalities of universal quantum groups
\begin{align*}
\cH( \f ) = \cO_{A(\f, N)}(GL)\quad \text{ and } \quad 
\cS(\f ) =  \cO_{A(\f, N)}(SL).
\end{align*}
\end{enumerate}
\end{theorem}

\begin{proof}
We only verify the statement for the left coaction of $\cH(\e)$ on $A(\e,N)$, and the arguments are analogous for the remaining items. Take $A:=A(\e ,N)=TV^*/\partial^{m-N}(\kk \e^*)$ with degree one generating space $V^*$. Let $K$ be a Hopf algebra coacting on $A$ from the left via $\rho_K: A\to K\otimes A$ induced by $\rho_K|_{V^*}: V^*\to K\otimes V^*$. By \cref{T:CoactionA(e)}((a)$\Rightarrow$(c)), $\kk \e^*$ is a $K$-subcomodule of $(V^*)^{\otimes m}$. Then, the universal property of $\cH(\e)$ in \cref{P:UH(e)}(d) ensures that there is a unique Hopf algebra map $\phi: \cH(\e)\to K$ satisfying 
\begin{align}\label{E:ManinQG1}
\rho_K|_{V^*}=(\phi\otimes \mathrm{id}_{V^*})\rho_\cH|_{V^*},
\end{align}
where $\rho_\cH: A\to \cH(\e)\otimes A$ is the coaction of $\cH(\e)$ given by $\rho_\cH(y_i)=\sum_{j=1}^n a_{ij}\otimes y_j$. Hence we have 
\begin{align}\label{E:ManinQG2}
\rho_K=(\phi\otimes \mathrm{id}_A)\rho_\cH
\end{align}
since both sides of \eqref{E:ManinQG2} are algebra maps from $A$ to $K\otimes A$ and  since they are equal when restricted to the generating space $V^*$ of $A$ due to \eqref{E:ManinQG1}. The Hopf algebra map $\phi: \cH(\e)\to K$  in \eqref{E:ManinQG2} is clearly unique. So, by definition, $\cH(\e)=\cO_{A(\e,N)}(GL)$ as Hopf algebras.
\end{proof}


 \section{The quantum group $\cH( \e ,  \f )$ associated to a pair of preregular forms $ \e $ and $ \f $} \label{sec:Hef}

This brings us to the heart of this work: the study of a universal quantum group $\cH(\e,\f)$ associated to preregular forms $\e$ and $\f$. We define this Hopf algebra in Section~\ref{def:Hef} and study its simultaneous coactions on superpotential algebras $A(\e,N)$ and $A(\f,N)$ in Section~\ref{sec:Hef-univ}. Then in view of Remark~\ref{R:Bialgebra}, we define and examine an analogous bialgebra that coacts simultaneously on $A(\e,N)$ and $A(\f,N)$ in the following section. The presentation of $\cH(\e,\f)$ as a pushout of $\cH(\e)$ and $\cH(\f)$ is then provided in Section~\ref{sec:pushout}. The section ends with a discussion of properties of $\cH(\e,\f)$ and related Hopf algebras in Section~\ref{sec:properties}.

\subsection{Presentation of $\cH( \e , \f )$ via generators and relations} \label{sec:Hef-pres}
\begin{definition}[$\cH( \e , \f )$] \label{def:Hef} 
For a pair of preregular forms $\e$ and $\f$, let $\cH(\e,\f)$ be the algebra with generators $(u_{ij})_{1 \leq i,j \leq n}$ and ${\sf D_e}^{\pm 1}$, ${\sf D_f}^{\pm 1}$ satisfying the relations
\begin{align}
\ts \sum_{i_1, \dots, i_m = 1}^n  \e _{i_1 \dots i_m} u_{i_1 j_1} \cdots u_{i_m j_m} 
~&=~  \e _{j_1 \dots j_m} {\sf D_e}, ~~\quad \quad \forall\, 1\le j_1,\dots, j_m\le n, \label{PresE}\\
\ts \sum_{i_1, \dots, i_m = 1}^n  \f _{i_1 \dots i_m} u_{j_1 i_1} \cdots u_{j_m i_m} 
~&=~  \f _{j_1 \dots j_m} {\sf D_f^{-1}}, \quad \quad \forall\, 1\le j_1,\dots, j_m\le n, \label{PresF}\\
{\sf D_e}{\sf D_e}^{-1} ~=~ {\sf D_e}^{- 1} {\sf D_e} ~&=~ {\sf D_f}{\sf D_f}^{-1} ~=~ {\sf D_f}^{- 1} {\sf D_f} ~=~ 1_{\mathcal{H}}. \label{Hef-Dinv}
\end{align}
\end{definition}

\begin{lemma}\label{L:InverseA}
For any $\tilde \e\in \mathrm{Aff}(\e)$ and $\tilde \f\in \mathrm{Aff}(\f)$, consider the following elements in $\cH(\e,\f)$
\begin{align*}
c_{ij}&=\ts \sum_{i_1, \dots, i_{m-1}, j_1, \dots, j_{m-1}=1}^n~~ {\sf D_e}^{-1} \widetilde{ \e }_{i i_1 \cdots i_{m-1}} u_{j_1 i_1} \cdots u_{j_{m-1} i_{m-1}}  \e _{j_1 \cdots j_{m-1} j},\\
d_{ij}&=\ts \sum_{i_1, \dots, i_{m-1}, j_1, \dots, j_{m-1}=1}^n ~~ \f _{i i_1 \cdots i_{m-1}} u_{j_1 i_1} \cdots u_{j_{m-1} i_{m-1}} \widetilde{ \f }_{j_1 \cdots j_{m-1} j} {\sf D_f},
\end{align*}
for all $1\le i,j\le n$. Then $\mathbb C\mathbb U=\I= \mathbb U \mathbb D$, where $\mathbb U$, $\mathbb C$ and $\mathbb D$ are the matrices $(u_{ij})_{1\leq i,j \leq n}$, $(c_{ij})_{1\le i,j\le n}$ and $(d_{ij})_{1\le i,j\le n}$, respectively. Moreover, $\mathbb C=\mathbb D$ and they do not depend on the choice of $\tilde\e$ and $\tilde\f$ in $\mathrm{Aff}(\e)$ and $\mathrm{Aff}(\f)$, respectively.
\end{lemma}

\begin{proof}
As shown in the proof of  \cref{L:RA(e)}, we have $\mathbb C\mathbb U=\I$. Similarly, we get that $\mathbb U\mathbb D=\I$. Hence, $\mathbb C=\mathbb C(\mathbb U\mathbb D)=(\mathbb C\mathbb U)\mathbb D=\mathbb D$. Let $\e'\in \mathrm{Aff}(\e)$, and $\mathbb C'$ be the corresponding matrix $(c_{ij}')_{1\le i,j\le n}$. Since $\mathbb D$ is the right inverse of $\mathbb U$, $\mathbb C'=\mathbb C'(\mathbb U\mathbb D)=\I\mathbb D=(\mathbb C\mathbb U)\mathbb D=\mathbb C$. The uniqueness of $\mathbb D$ can be proved by using $\mathbb C$ as the left inverse of $\mathbb U$.
\end{proof}

\begin{proposition}\label{P:H(ef)}
The algebra $\cH(\e,\f)$ admits a Hopf algebra structure, with comultiplication $\Delta$ defined by  
\begin{align}
\Delta(u_{ij}) = \ts \sum_{k=1}^n u_{ik}\otimes  u_{kj},\quad \Delta({\sf D}_\e^{\pm 1})={\sf D}_\e^{\pm 1}\otimes {\sf D}_\e^{\pm 1},\quad \Delta({\sf D}_\f^{\pm 1})={\sf D}_\f^{\pm 1}\otimes {\sf D}_\f^{\pm 1}, \label{Hef-coprod}
\end{align}
with counit $\varepsilon$ defined by 
\begin{align}
\varepsilon(u_{ij})  ~&=~ \delta_{ij}, \quad \varepsilon({\sf D}_\e^{\pm 1})=\varepsilon({\sf D}_\f^{\pm 1})=1 \label{Hef-counit}
\end{align}
and with antipode $S$ defined by
\begin{align}
S(u_{ij}) ~&=~ \ts \sum_{i_1, \dots, i_{m-1}, j_1, \dots, j_{m-1}=1}^n~~ {\sf D_e}^{-1} \widetilde{ \e }_{i i_1 \cdots i_{m-1}} u_{j_1 i_1} \cdots u_{j_{m-1} i_{m-1}}  \e _{j_1 \cdots j_{m-1} j} \label{Hef-SA1}\\
~&=~ \ts \sum_{i_1, \dots, i_{m-1}, j_1, \dots, j_{m-1}=1}^n ~~ \f _{i i_1 \cdots i_{m-1}} u_{j_1 i_1} \cdots u_{j_{m-1} i_{m-1}} \widetilde{ \f }_{j_1 \cdots j_{m-1} j} {\sf D_f}, \label{Hef-SA2}
\end{align}
for $1\le i,j\le n$ and for any choice of  $\tilde \e\in \mathrm{Aff}(\e)$, $\tilde \f\in \mathrm{Aff}(\f)$, with $S({\sf D_e}^{\pm 1})={\sf D_e}^{\mp 1}$, $S({\sf D_f}^{\pm 1})={\sf D_f}^{\mp 1}$. 
\end{proposition}

\begin{proof}
One sees easily that \eqref{Hef-coprod} and \eqref{Hef-counit} equip $\cH(\e,\f)$ with a bialgebra structure. In view of \cref{L:InverseA}, $S(u_{ij})$ is well-defined in $\cH(\e,\f)$ and
\begin{align*}
\ts \sum_{k=1}^n u_{ik}S(u_{kj})=\sum_{k=1}^n S(u_{ik})u_{kj}=\varepsilon(u_{ij}),
\end{align*}
for all $1\le i,j\le n$. Hence it suffices to show that $S$ preserves the relations \eqref{PresE}-\eqref{Hef-Dinv}. For \eqref{PresE}, consider 
\begin{align*}
S(\ts \sum_{i_1,\dots,i_m=1}^n \e_{i_1\cdots i_m}& u_{i_1j_1}\cdots u_{i_mj_m})=\ts \sum_{i_1,\dots,i_m=1}^n\e_{i_1\cdots i_m} S(u_{i_mj_m})\cdots S(u_{i_1j_1})\\
&= \ts \sum_{i_1,\dots,i_m=1}^n(\sum_{r_1,\dots,r_m=1}^n {\sf D}_\e^{-1} \e_{r_1\cdots r_m}u_{r_1i_1}\cdots u_{r_mi_m})\, S(u_{i_mj_m})\cdots S(u_{i_1j_1})\\
&=\ts \sum_{r_1,\dots,r_m=1}^n {\sf D}_\e^{-1} \e_{r_1\cdots r_m}(\sum_{i_1,\dots,i_m=1}^n u_{r_1i_1}\cdots (u_{r_mi_m}S(u_{i_mj_m}))\cdots S(u_{i_1j_1}))\\
&=\ts \sum_{r_1,\dots,r_m=1}^n \delta_{r_1j_1}\cdots \delta_{r_mj_m}{\sf D}_\e^{-1} \e_{r_1\cdots r_m}={\sf D}_\e^{-1} \e_{j_1\cdots j_m}=S(\e_{j_1\cdots j_m}{\sf D}_\e).
\end{align*}
Similarly, the relation \eqref{PresF} is preserved  and the preservation of \eqref{Hef-Dinv} is obvious. Thus, $S$ is the antipode of $\cH(\e,\f)$. 
\end{proof}

\begin{corollary}\label{C:QHef}
We have two quotient Hopf algebra maps 
\[\phi_\e: \cH(\e)\to \cH(\e,\f)\quad \text{and}\quad \phi_\f: \cH(\f)\to \cH(\e,\f)\] 
defined by $\phi_\e(a_{ij})=u_{ij}$, $\phi_\e(b_{ij})=S(u_{ij})$, $\phi_\e({\sf D}_\e^{\pm 1})={\sf D}_\e^{\pm 1}$ and $\phi_\f(a_{ij})=u_{ij}$, $\phi_\f(b_{ij})=S(u_{ij})$, $\phi_\f({\sf D}_\f^{\pm 1})={\sf D}_\f^{\pm 1}$, respectively.  \qed
\end{corollary}

We denote by $\eta'_{\sf D_{\sf e}}, \eta'_{\sf D_{\sf f}}$ and $\eta'_{\sf e},\eta'_{\sf f}$ four automorphisms of $\cH(\e,\f)$ satisfying 
\begin{align*}
\eta'_{\sf D_{\sf e}}(\mathbb U)&={\sf D_{\sf e}}\mathbb U{\sf D_{\sf e}}^{-1},\hspace{0.5in} \eta'_{\sf D_{\sf e}}({\sf D}_{\sf e}^{\pm 1})={\sf D}_{\sf e}^{\pm 1},\hspace{0.5in} \eta'_{\sf D_{\sf e}}({\sf D}_{\sf f}^{\pm 1})={\sf D_{\sf e}}{\sf D}_{\sf f}^{\pm 1}{\sf D_{\sf e}}^{-1};\\
\eta'_{\sf D_{\sf f}}(\mathbb U)&={\sf D_{\sf f}}\mathbb U{\sf D_{\sf f}}^{-1},\hspace{0.52in} \eta'_{\sf D_{\sf f}}({\sf D}_{\sf f}^{\pm 1})={\sf D}_{\sf f}^{\pm 1},\hspace{0.5in} \eta'_{\sf D_{\sf f}}({\sf D}_{\sf e}^{\pm 1})={\sf D_{\sf f}}{\sf D}_{\sf e}^{\pm 1}{\sf D_{\sf f}}^{-1};\\
\eta'_{\e}(\mathbb U)&=\mathbb P^{-1}\mathbb U\mathbb P,\hspace{0.72in} \eta'_\e({\sf D}_{\sf e}^{\pm 1})={\sf D}_{\sf e}^{\pm 1},\hspace{0.58in} \eta'_\e({\sf D}_{\sf f}^{\pm 1})={\sf D}_{\sf f}^{\pm 1};\\
\eta'_{\f}(\mathbb U)&=\mathbb Q^{-T}\mathbb U\mathbb Q^T,\hspace{0.56in} \eta'_\f({\sf D}_{\sf e}^{\pm 1})={\sf D}_{\sf e}^{\pm 1},\hspace{0.58in} \eta'_\f({\sf D}_{\sf f}^{\pm 1})={\sf D}_{\sf f}^{\pm 1}.
\end{align*}
By \cref{P:H(ef)}, it is clear that they are all Hopf algebra automorphisms of $\cH(\e,\f)$.

\begin{corollary}\label{Hef-invol} 
We have $\eta'_{\sf D_{\sf e}}\circ S^2=\eta'_{\sf e}$ and  $\eta'_{\sf D_{\sf f}}\circ S^2=\eta'_{\sf f}$ in $\cH(\e,\f)$. Thus, the antipode $S$ of $\cH(\e,\f)$ is bijective. Moreover, the Hopf algebra $\cH( \e ,  \f )$ is involutory if one of the conditions below is satisfied: 
\begin{enumerate}
\item ${\sf D}_ \e $ is central and the matrix $\bP$ is a scalar multiple of $\bI$; or
\item ${\sf D}_ \f $ is central and the matrix $\bQ$ is a scalar multiple of $\bI$.
\end{enumerate}
\end{corollary}

\begin{proof}
By \cref{C:QHef}, there is a surjection $\phi_{\e}: \cH(\e)\to \cH(\e,\f)$, which commutes with their coactions on $V$ and $M_\e$. Apply $\phi= \phi_{\e}$ to \cref{C:S2He}. Since $\phi\circ \eta_{{\sf D}_\e}=\eta'_{{\sf D}_\e}\circ \phi$ and $\phi\circ \eta_\e=\eta'_\e\circ \phi$, we obtain that $\eta'_{\sf D_{\sf e}}\circ S^2=\eta'_{\sf e}$ for $\cH(\e,\f)$. Likewise, $\phi = \phi_{\f}$  to get  $\eta'_{\sf D_{\sf f}}\circ S^2=\eta'_{\sf f}$ for $\cH(\e,\f)$. 

Now if $\mathbb P=\lambda \I$ for some nonzero $\lambda\in \kk$ and ${\sf D}_\e$ is central, it is clear that $\eta'_{{\sf D}_\e}=\eta'_\e=\mathrm{id}$, which implies that $S^2=\id$. Similarly, (b) implies that $S^2=\id$.
 \end{proof}

We end this section by discussing  universal properties of $\cH(\e,\f)$.

\begin{proposition}\label{P:UHef} 
\
\begin{enumerate}
\item Endow $V$ with the right $\cH(\e,\f)$-comodule structure defined by $\rho_\cH(v_j)=\sum_{i=1}^n v_i\otimes u_{ij}$. Then, $\kk \f^*$ is a right $\cH(\e,\f)$-subcomodule of $V^{\otimes m}$. 

\medskip

\item Endow $V^*$ with the induced left $\cH(\e,\f)$-comodule structure such that $\rho_\cH(\theta_i)=\sum_{j=1}^n u_{ij}\otimes \theta_j$. Then, $\kk \e^*$ is a left $\cH(\e,\f)$-subcomodule of $(V^*)^{\otimes m}$.

\medskip

\item Let $K$ be a Hopf algebra that right coacts on $V$ so that $\kk \f^*$ is a right $K$-subcomodule of $V^{\otimes m}$. Moreover, with the induced left $K$-coaction on $V^*$, $\kk \e^*$ is a left $K$-subcomodule of $(V^*)^{\otimes m}$. Then, there exists a unique Hopf algebra map $\phi: \cH(\e,\f)\to K$ such that $(\mathrm{id}_V\otimes \phi)\circ \rho_\cH=\rho_K$, where $\rho_\cH$ and $\rho_K$ denote the coactions on $V$ of $\cH(\e,\f)$ and $K$, respectively. 

\medskip

\item Endow the three vector spaces $V$, $M_\e$ and $\kk \f^*$ with the right $\cH(\e,\f)$-comodule structures defined by $\rho_\cH(v_j)=\sum_{i=1}^n v_i\otimes u_{ij}$, $\rho_\cH(m_\e)=m_\e \otimes {\sf D}_\e$ and $\rho_\cH(\f^*)=\f^*\otimes {\sf D}_\f^{-1}$, respectively. Then, the following 
\[
\xymatrix{
\kk \f^*\ar@{^{(}->}[rr] & &  V^{\otimes m} \ar[rr]^-{\hat e} & & M_\e
}
\]
is a sequence of right $\cH(\e,\f)$-comodule maps. 

\medskip

\item Let $K$ be a Hopf algebra that right coacts on $V$, on $M_\e$ and on $\kk \f^*$ so that the sequence above is a sequence of right $K$-comodule maps. Then, there exists a unique Hopf algebra map $\phi: \cH(\e,\f)\to K$ such that $(\mathrm{id}_V\otimes \phi)\circ \rho_\cH=\rho_K$, where $\rho_\cH$ and $\rho_K$ denote the coactions on $V$ of $\cH(\e,\f)$ and $K$, respectively.

\medskip

\item Endow the three vector spaces $V^*$, $M_\f$ and $\kk \e^*$ with the left $\cH(\e,\f)$-comodule structures defined by $\rho_\cH(\theta_i)=\sum_{j=1}^n u_{ij}\otimes \theta_j$, $\rho_\cH(m_\f)={\sf D}_\f^{-1}\otimes m_\f$ and $\rho_\cH(\e^*)={\sf D}_\e\otimes \e^*$, respectively. Then, the following
\[
\xymatrix{
\kk \e^*\ar@{^{(}->}[rr] & &  (V^*)^{\otimes m} \ar[rr]^-{\hat f} & & M_\f
}
\]
is a sequence of left $\cH(\e,\f)$-comodule maps.

\medskip

\item Let $K$ be a Hopf algebra that left coacts on all $V^*$, $M_\f$ and $\kk \e^*$ so that the  sequence above is a sequence of $K$-comodule maps. Then, there exists a unique Hopf algebra map \linebreak $\phi: \cH(\e,\f)\to K$ such that $(\phi\otimes \mathrm{id}_{V^*})\circ \rho_\cH=\rho_K$, where $\rho_\cH$ and $\rho_K$ denote the coactions on $V^*$ of $\cH(\e,\f)$ and $K$, respectively. 
\end{enumerate}
\end{proposition}

\begin{proof}
We only verify for part (d,e); the rest follows similarly or is straightforward. 

\smallskip

(d) By \cref{T:CoactionA(e)}((e)$\Rightarrow$(d)) and \eqref{PresE}, one sees that $\hat \e: V^{\otimes m}\to M_\e$ is a right $\cH(\e,\f)$-comodule map with the $\cH(\e,\f)$-coactions on $V$ and $M_\e$ given above. Similarly, we get $\kk\f^*\hookrightarrow V^{\otimes m}$ is a left $\cH(\e,\f)$-comodule map by \cref{T:CoactionA(f)}((e)$\Rightarrow$(c)).

\smallskip

(e) Let $K$ be a Hopf algebra that right coacts on $V$, $M_\e$ and $\kk \f^*$ via $\rho_K(v_j)=\sum_{i=1}^n v_i\otimes k_{ij}$, $\rho_K(m_\e)=m_\e \otimes d$ and $\rho_K(\f^*)=\f^*\otimes d'$ respectively, where $k_{ij}\in K$ and $d,d'\in K$ are two grouplike elements. Suppose the sequence of maps in part (d) is a sequence of right $K$-comodule maps. Then, by \cref{T:CoactionA(e)}((d)$\Rightarrow$(e)) and \cref{T:CoactionA(f)}((c)$\Rightarrow$(e)), there exists a Hopf algebra map $\phi: \cH(\e,\f)\to K$ given by $\phi(u_{ij})=k_{ij}$, $\phi({\sf D}_\e^{\pm 1})=d^{\pm 1}$ and $\phi({\sf D}_\f^{\pm 1})=d'^{\mp 1}$. It is clear that $\phi$ is the unique map satisfying $(\mathrm{id}_V\otimes \phi)\circ \rho_\cH=\rho_K$. This proves our result.  
\end{proof}


\subsection{Balanced coaction on a pair of superpotential algebras $A(\e,N)$ and $A(\f,N)$} \label{sec:Hef-univ}
Now we discuss universal coactions on a pair of superpotential algebras $A( \e ,N)$  and $A( \f ,N)$. To begin, consider the following Hopf quotient of $\cH(\e,\f)$.

\begin{definition}[$\cS(\e,\f)$]\label{D:S(e,f)}
Take $\cS(\e,\f):=\cH(\e,\f)/({\sf D}_\e-1,{\sf D}_\f-1)$, a Hopf quotient of $\cH(\e,\f)$.  
\end{definition}

In $\cS(\e,\f)$, we will still keep the notation $(u_{ij})_{1\le i,j\le n}$ for its image under the quotient map $\cH(\e,\f)\twoheadrightarrow \cS(\e,\f)$, subject to \eqref{PresE}, \eqref{PresF}, with ${\sf D}_\e^{\pm 1}={\sf D}_\f^{\pm 1}=1$. Moreover, $\cS(\e,\f)$ has similar universal properties as those of $\cH(\e,\f)$ in \Cref{P:UHef} by requiring the corresponding coaction on $M_\e$, $M_\f$ or $\kk \e^*$, $\kk \f^*$ to be trivial.

\begin{definition}[balanced coaction] \label{D:Sym}
Let $K$ be a Hopf algebra that simultaneously left coacts on $A(\e,N)$ and right coacts on  $A(\f,N)$. We say that the $K$-coactions are \emph{balanced} if we have elements $(k_{ij})_{1\le i,j\le n}\in K$ such that 
$$\ts \rho_K^\ell(x_i)=\sum_{j=1}^n k_{ij}\otimes x_j \quad \text{and} \quad \rho_K^r(y_j)=\sum_{i=1}^n y_i\otimes k_{ij},$$ where $\rho_K^\ell$ denotes the left $K$-coaction on $A(\e,N)$ and $\rho_K^r$ denotes the right $K$-coaction on $A(\f,N)$. 
\end{definition}

\begin{proposition}\label{Hef-coacts}
We have the following inner-faithful balanced coactions:
\begin{enumerate}
\item $\cH( \e , \f )$ left coacts on $A( \e ,N)$ and right coacts on $A( \f ,N)$ via 
$$\rho^\ell_\cH(x_i) = \textstyle \sum_{j=1}^n u_{ij} \otimes x_j
\quad \text{ and } \quad
\rho^r_\cH(y_j) = \textstyle \sum_{i=1}^n y_i \otimes u_{ij}.$$
Moreover, the homological codeterminant of the left \textnormal{(}resp., right\textnormal{)} $\cH( \e , \f )$-coaction on $A( \e ,N)$ \textnormal{(}resp., $A( \f ,N)$\textnormal{)} is given by ${\sf D}_ \e $ \textnormal{(}resp., ${\sf D}_ \f $\textnormal{)}.
\item $\cS( \e , \f )$ left coacts on $A( \e ,N)$ and right coacts on $A( \f ,N)$ via 
$$\rho^\ell_\cS(x_i) = \textstyle \sum_{j=1}^n u_{ij} \otimes x_j
\quad \text{ and } \quad
\rho^r_\cS(y_j) = \textstyle \sum_{i=1}^n y_i \otimes u_{ij}.$$
Moreover, these $\cS(\e,\f)$-coactions both have trivial homological codeterminant. 
\end{enumerate}
\end{proposition}

\begin{proof}
By using \eqref{PresE} and~\eqref{PresF}, this follows from \cref{T:CoactionA(e)}((e)$\Rightarrow$(a)) and \cref{T:CoactionA(f)}((e)$\Rightarrow$(a)).
\end{proof}

\begin{definition}[balanced general/special linear group, $\mathcal{O}_{A( \e , N), A( \f , N)}(GL)$, $\mathcal{O}_{A( \e , N), A( \f , N)}(SL)$] \label{OA(e)A(f)}
\
\begin{enumerate}
\item
The {\it balanced quantum general linear group associated to a pair of $N$-Koszul AS regular algebras $A(\e, N)$ and $A(\f, N)$}, denoted by $\mathcal{O}_{A( \e , N), A( \f , N)}(GL)$, is defined to be the Hopf algebra that left coacts on $A(\e, N)$ and right coacts on $A(\f, N)$ in a balanced fashion satisfying the following universal property: for any Hopf algebra $K$ that left coacts on $A( \e ,N)$ and right coacts on $A( \f ,N)$ in a balanced manner, there exists a unique Hopf algebra map $\phi: \mathcal{O}_{A( \e , N), A( \f , N)}(GL)\to K$ such that 
$$(\phi\otimes \mathrm{id}_{\mathcal{O}_{A( \e , N), A( \f , N)}(GL)})\circ \rho_\cO^\ell=\rho_K^\ell \quad \text{and} \quad (\mathrm{id}_{\mathcal{O}_{A( \e , N), A( \f , N)}(GL)}\otimes \phi)\circ \rho_\cO^r=\rho_K^r,$$ where $\rho_\cO^\ell$ and $\rho_K^\ell$ (resp. $\rho_\cO^r$ and $\rho_K^r$) denote the left (resp. right) coactions on $A(\e,N)$ (resp. $A(\f,N)$) of $\mathcal{O}_{A( \e , N), A( \f , N)}(GL)$ and $K$, respectively. 

\medskip

\item Similarly, we can define the {\it balanced quantum special linear group associated to a pair of $N$-Koszul AS regular algebras $A(\e, N)$ and $A(\f, N)$} by requiring all the coactions in part (a) to have trivial homological codeterminant. We denoted it by $\mathcal{O}_{A( \e , N), A( \f , N)}(SL)$.
\end{enumerate}
\end{definition}

\begin{theorem} \label{H-Oisom}
\
\begin{enumerate}
\item The balanced coactions of $\cH(\e,\f)$ on $A(\e,N)$ and $A(\f,N)$ given in \cref{Hef-coacts}(a) satisfy the universal property described by \cref{OA(e)A(f)}(a). Therefore, we have the following equality of universal quantum groups
\begin{align*}
\cH(\e,\f)= \mathcal{O}_{A( \e , N), A( \f , N)}(GL).
\end{align*}
\item The balanced coactions of $\cS(\e,\f)$ on $A(\e,N)$ and $A(\f,N)$ given in \cref{Hef-coacts}(b) satisfy the universal property described by \cref{OA(e)A(f)}(b). Therefore, we have the following equality of universal quantum groups
\begin{align*}
\cS(\e,\f)= \mathcal{O}_{A( \e , N), A( \f , N)}(SL).
\end{align*}
\end{enumerate}
\end{theorem}

\begin{proof}
We only verify for part (a); the argument for part (b) is analogous. Let $K$ be a Hopf algebra that left coacts on $A(\e,N)$ and right coacts on $A(\f,N)$ via $\rho_K^\ell(x_i)=\sum_{j=1}^n k_{ij}\otimes x_j$ and  $\rho_K^r(y_j)=\sum_{i=1}^n y_i\otimes k_{ij}$ for some $k_{ij}\in K$. By \cref{T:CoactionA(e)}((a)$\Rightarrow$(e)), there is a unique grouplike element $d\in K$ satisfying  
\begin{align*}
\ts \sum_{i_1,\dots,i_m=1}^n \e_{i_1\dots i_m}\, k_{i_1j_1}\cdots k_{i_mj_m}=\e_{j_1\cdots j_m}d,
\end{align*} 
for all $1\le j_1,\dots,j_m\le n$. Similarly, by \cref{T:CoactionA(f)}((a)$\Rightarrow$(e)), there is a unique grouplike element $d'\in K$ satisfying  
\begin{align*}
\ts \sum_{i_1,\dots,i_m=1}^n \f_{i_1\dots i_m}\, k_{j_1i_1}\cdots k_{j_mi_m}=\f_{j_1\cdots j_m}d',
\end{align*} 
for all $1\le j_1,\dots,j_m\le n$. Thus, there is a Hopf algebra map $\phi: \cH(\e,\f)\to K$ via $\phi(u_{ij})=k_{ij}$, $\phi({\sf D}_\e^{\pm 1})=d^{\pm 1}$ and $\phi({\sf D}_\f^{\pm 1})=d'^{\mp 1}$. Hence, $(\phi\otimes \mathrm{id}_{A(\e,N)})\circ \rho_\cH^\ell=\rho_K^\ell$ since they are equal on the generators $x_i$'s of $A(\e,N)$. 
Likewise, we  have $(\mathrm{id}_{A(\f,N)}\otimes \phi)\circ \rho_\cH^r=\rho_K^r$. Finally, the uniqueness of $\phi$ is clear. This proves our result. 
\end{proof}

\subsection{Bialgebras associated to a pair of preregular forms $\e$ and $\f$} In light of \Cref{R:Bialgebra}, we consider the following bialgebra.

\begin{definition}[$\mathcal{O}_{A( \e , N), A( \f , N)}(M)$]
For a pair of $N$-Koszul AS regular algebras $A(\e,N)$ and $A(\f,N)$, let $\mathcal{O}_{A( \e , N), A( \f , N)}(M)$ be the algebra with generators $(z_{ij})_{1\le i,j\le n}$ satisfying the relations
\begin{align*}
\ts \sum_{i_1,\dots,i_N,j_1,\dots,j_N=1}^n \e_{\lambda_1\cdots \lambda_{m-N}i_1\cdots i_N}\mu_{j_1\cdots j_N} z_{i_1j_1}\cdots z_{i_Nj_N}&=0, \text{and}\\
\ts \sum_{i_1,\dots,i_N,j_1,\dots,j_N=1}^n \f_{\lambda_1\cdots \lambda_{m-N}j_1\cdots j_N}\nu_{i_1\cdots i_N} z_{i_1j_1}\cdots z_{i_Nj_N}&=0,
\end{align*}
for all $\lambda_1,\dots \lambda_{m-N}\in\{1,\dots,n\}$, $\mu\in \partial^{m-N}(\kk \e^*)^\perp$ and $\nu \in \partial^{m-N}(\kk \f^*)^\perp$. Moroever, we have a bialgebra structure on $\mathcal{O}_{A( \e , N), A( \f , N)}(M)$ by defining $\Delta(z_{ij})=\sum_{k=1}^n z_{ik}\otimes z_{kj}$ and $\varepsilon(z_{ij})=\delta_{ij}$ for all $1\le i,j\le n$.
\end{definition}

\begin{proposition}\label{P:QMSHef}
The bialgebra $\mathcal{O}_{A( \e , N), A( \f , N)}(M)$ left coacts on $A(\e,N)$ and right coacts on $A(\f,N)$ in a balanced fashion via 
$$\ts \rho_\cO^\ell(x_i)=\sum_{j=1} z_{ij}\otimes x_j \quad \text{and} \quad \rho_\cO^r(y_j)=\sum_{i=1}^n y_i\otimes z_{ij},$$ respectively. Moreover,  the coactions satisfy the following universal property: for any bialgebra $B$ that left coacts on $A(\e,N)$ and right coacts on $A(\f,N)$ in a balanced way, there exists a unique bialgebra map $\phi: \mathcal{O}_{A( \e , N), A( \f , N)}(M)\to B$ such that 
$$(\phi\otimes \mathrm{id}_{A(\e,N)})\circ \rho_\cO^\ell=\rho_B^\ell \quad \text{and} \quad (\mathrm{id}_{A(\f,N)}\otimes \phi)\circ \rho_\cO^r=\rho_B^r,$$ where $\rho_\cO^\ell$ and $\rho_B^\ell$ \textnormal{(}resp., $\rho_\cO^r$ and $\rho_B^r$\textnormal{)} denote the left \textnormal{(}resp., right\textnormal{)} coactions on $A(\e,N)$ \textnormal{(}resp.,~$A(\f,N)$\textnormal{)} of $\mathcal{O}_{A( \e , N), A( \f , N)}$ and $B$, respectively. 
\end{proposition}

\begin{proof}
Apply the same proof of \cref{H-Oisom} by using \cref{R:Bialgebra}. 
\end{proof}

According to  \cref{R:Bialgebra} (or Theorem~\ref{T:CoactionA(e)}((b)$\Rightarrow$(e)) and Theorem~\ref{T:CoactionA(f)}((b)$\Rightarrow$(e))), we have two unique grouplike elements in $\mathcal{O}_{A( \e , N), A( \f , N)}(M)$ given by

\[
g_1:=\ts \frac{1}{\e_{j_1\cdots j_m}}\sum_{i_1,\dots,i_m=1}^n \e_{i_1\cdots i_m}\, z_{i_1j_1}\cdots z_{i_mj_m}, \quad \quad   g_2:=\frac{1}{\f_{j_1\cdots j_m}}\sum_{i_1,\dots,i_m=1}^n \f_{i_1\cdots i_m}\, z_{j_1i_1}\cdots z_{j_mi_m},
\]

\smallskip

\noindent for any choice of $j_1,\dots,j_m\in \{1,\dots,n\}$ satisfying $\e_{j_1\cdots j_m}$, $\f_{j_1\cdots j_m}\neq 0$.

\begin{theorem}\label{T:LocalBi}
We have the following isomorphisms of universal quantum linear groups. 
\begin{enumerate}
\item $\mathcal{O}_{A( \e , N), A( \f , N)}(GL)\cong \mathcal{O}_{A( \e , N), A( \f , N)}(M)G^{-1}$, when the multiplicative set $G=\langle g_1,g_2\rangle$ is an Ore set of $\mathcal{O}_{A( \e , N), A( \f , N)}(M)$.
\smallskip
\item $\mathcal{O}_{A( \e , N), A( \f , N)}(SL)\cong \mathcal{O}_{A( \e , N), A( \f , N)}(M)/(g_1-1,g_2-1)$.
\end{enumerate}
\end{theorem}

\begin{proof}
We only prove for part (a), and part (b) can be proved in the same fashion. By \cref{H-Oisom}, we can take $\mathcal{O}_{A( \e , N), A( \f , N)}(GL)=\cH(\e,\f)$. We apply \cref{P:QMSHef} to the balanced coactions of $\cH(\e,\f)$ on $A(\e,N)$ and $A(\f,N)$ in \cref{Hef-coacts}. Hence there is a bialgebra map $\phi: \mathcal{O}_{A( \e , N), A( \f , N)}(M)\to \cH(\e,\f)$ via $\phi(z_{ij})=u_{ij}$. One checks that $\phi(g_1)={\sf D}_\e$ and $\phi(g_2)={\sf D}_\f^{-1}$. By assumption, $\phi$ naturally extends to the localization of $\mathcal{O}_{A( \e , N), A( \f , N)}(M)$ at $G$, where we still write $\phi: \mathcal{O}_{A( \e , N), A( \f , N)}(M)G^{-1}\to \cH(\e,\f)$. Note that $\phi$ has an inverse given by $\phi^{-1}(u_{ij})=z_{ij}$ and $\phi^{-1}({\sf D}_\e^{\pm 1})=g_1^{\pm 1}, \phi^{-1} ({\sf D}_\f^{\pm 1})=g_2^{\mp 1}$. Therefore, $\phi: \mathcal{O}_{A( \e , N), A( \f , N)}(M)G^{-1}\to \cH(\e,\f)$ is an isomorphism of bialgebras. Finally, $\phi^{\pm 1}$ equips $\mathcal{O}_{A( \e , N), A( \f , N)}(M)G^{-1}$ with a unique Hopf algebra structure by \Cref{P:H(ef)}. 
\end{proof}
  
 \subsection{Presentations of $\cH( \e , \f )$ and $\cS(\e,\f)$ via pushouts} \label{sec:pushout}
In this part, we realize the Hopf algebra $\cH( \e , \f )$ as a pushout in the category of Hopf algebras, by making use of its universal property.

\begin{definition} \label{D:pushout}
Let $A,B,C$ be three Hopf algebras together with Hopf algebra maps $f: C\to A$ and $g:C\to B$. The \emph{pushout} of $A$ and $B$ along $C$ is a Hopf algebra, denoted by $A\amalg_C B$, together with Hopf algebra maps $\alpha: A\to A\amalg_C B$ and $\beta: B\to A\amalg_C B$ so that $\alpha f=\beta g$ satisfying the following condition. For any Hopf algebra $D$ with Hopf algebra maps $p:A\to D$ and $q: B\to D$ with $pf=qg$, there is a unique Hopf algebra map $\Psi: A\amalg_C B\to D$ with $p=\Psi\alpha$ and $q=\Psi\beta$.
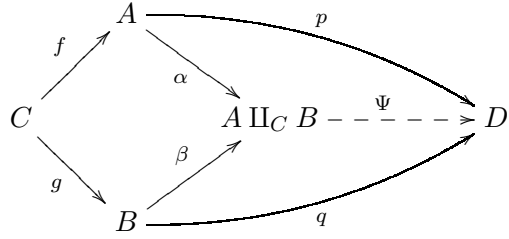
\[
\xymatrix{
 &  A\ar[dr]_-{\alpha}\ar@/^1pc/[rrrd]^-{p} &     &&\\
C\ar[ur]^-{f}\ar[dr]_-{g} &    & A\amalg_C B\ar@{-->}[rr]^\Psi  && D\\
  &  B\ar[ru]^-{\beta}\ar@/_1pc/[rrru]_-{q} &&
}
\]
\end{definition}

The following result pertains to the existence of pushouts of Hopf algebras.

\begin{proposition} \cite[Theorem~2.2]{agore}  \cite{porst} \label{P:PushA}
A pushout of Hopf algebras  exists in the category of Hopf algebras,  and it is preserved under the forgetful functor from the category of Hopf algebras to the category of algebras.  \qed
\end{proposition}

Let us define a Hopf algebra that will play the role of $C$ in Definition~\ref{D:pushout}: the free Hopf algebra on the $n\times n$ matrix coalgebra spanned by $(t_{ij})_{1 \leq i,j \leq n}$.
Take $t_{ij}=:t_{ij}^1$ and consider the Hopf algebra below.

\begin{definition}[$\cF$] \cite{tak_free} \label{D:FreeH}
Let $\mathcal F$ be the Hopf algebra with generators $(t_{ij}^\ell)_{1\le i,j\le n, ~\ell \ge 1}$ satisfying the relations 
\[
\begin{cases}
\smallskip
\ts \sum_{k=1}^n t^\ell_{ik}t^{\ell+1}_{kj}=\ts \sum_{k=1}^n t_{ik}^{\ell+1}t_{kj}^\ell=\delta_{ij}, &\quad \text{for } \ell \text{ odd} \\

\ts \sum_{k=1}^n t^\ell_{kj}t^{\ell+1}_{ik}=\ts \sum_{k=1}^n t^{\ell+1}_{kj}t^{\ell}_{ik}=\delta_{ij}, &\quad \text{for } \ell \text{ even}
\end{cases}
\]
for all $1\le i,j\le n$ and $\ell\ge 1$, with comultiplication $\Delta$ and counit $\varepsilon$ defined by 
\[
\Delta(t_{ij}^\ell)=
\begin{cases}
\ts \sum_{k=1}^n t_{ik}^\ell \otimes t_{kj}^\ell &   \quad \text{for } \ell \text{ odd}\\
\ts \sum_{k=1}^n t_{kj}^\ell \otimes t_{ik}^\ell & \quad \text{for } \ell \text{ even}
\end{cases}
,\quad \quad \varepsilon(t^\ell_{ij})=\delta_{ij},
\]
and with antipode $S$ defined by $S(t_{ij}^\ell)=t_{ij}^{\ell+1}$ for all $1\le i,j\le n$ and $\ell\ge 1$. 
\end{definition}

The result below is clear.

\begin{lemma} \label{L:Funiv}
The Hopf algebra $\mathcal F$  right coacts universally on $V$  via  $\rho_{\mathcal F}(v_j)=\sum_{i=1}^n v_i\otimes t_{ij}$, that is, for any Hopf algebra $K$ that right coacts on $V$ via $\rho_K$, there is a unique Hopf algebra $\phi: \mathcal F\to K$ such that $(\mathrm{id}_V\otimes \phi)\circ \rho_{\mathcal F}=\rho_K$.  \qed
\end{lemma}

In the following, we denote by $\rho_\e$ the universal right $\cH(\e)$-coactions  on $V$ and $M_\e$ described in \cref{P:UH(e)}(a); and by $\rho_\f$ the universal right $\cH(\f)$-coaction on $V$ described in \cref{P:UH(f)}(c). Thus by Lemma~\ref{L:Funiv}, there are two unique Hopf algebra maps $f: \mathcal F\to \cH(\e)$ and $g: \mathcal F\to \cH(\f)$ satisfying 
\begin{align}\label{E:Push1}
(\mathrm{id}_V\otimes f)\rho_{\mathcal F}=\rho_\e\quad \text{and}\quad  (\mathrm{id}_V\otimes g)\rho_{\mathcal F}=\rho_\f.
\end{align}
Moreover, we denote by $\rho_\cH$ the universal right $\cH(\e,\f)$-coactions on $V$ and $M_\e$ given in \cref{P:UHef}(d). Hence by \cref{P:UH(e)}(b) and \cref{P:UH(f)}(d), there are unique Hopf algebra maps $\alpha: \cH(\e)\to \cH(\e,\f)$ and $\beta: \cH(\f)\to \cH(\e,\f)$ satisfying 
\begin{align}\label{E:Push2}
(\mathrm{id}_V\otimes \alpha)\rho_\e=\rho_\cH\quad \text{and}\quad  (\mathrm{id}_V\otimes \beta)\rho_\f=\rho_\cH,
\end{align}
which are explicitly given in \Cref{C:QHef} as $\alpha=\phi_\e$ and $\beta=\phi_\f$.

\begin{theorem}\label{push}
Retain the notation above. We have $\alpha f=\beta g$. Moreover, $\cH(\e,f)$ is the pushout of $\cH(\e)$ and $\cH(\f)$ along $\mathcal F$ as in the following diagram. 
\begin{align}\label{eq:push}
\begin{tabular}{c}
\xymatrix{
 &  \cH(\e)\ar[dr]_-{\alpha}\ar@/^1pc/[rrrd]^-{p} &     &&\\
\mathcal F\ar[ur]^-{f}\ar[dr]_-{g} &    & \cH(\e,\f)\ar@{-->}[rr]^\Psi  && K\\
  &  \cH(\f)\ar[ru]^-{\beta}\ar@/_1pc/[rrru]_-{q} && \\
}
\end{tabular}
\end{align}
\end{theorem}

\begin{proof}
By \eqref{E:Push1} and \eqref{E:Push2}, we have $\rho_\cH=(\mathrm{id}_V\otimes \alpha f)\rho_{\mathcal F}$. Similarly, we obtain $\rho_\cH=(\mathrm{id}_V\otimes \beta g)\rho_{\mathcal F}$. Thus the universal property of $\mathcal F$ implies that $\alpha f=\beta g$. 

Let $K$ be a Hopf algebra with two Hopf algebra maps $p: \cH(\e)\to K$ and $q: \cH(\f)\to K$ satisfying $pf=qg$. Consider the $K$-coaction on $V$ induced by 
\begin{align*}
(\mathrm{id}_V\otimes p)\rho_\e \overset{\eqref{E:Push1}}{=} (\mathrm{id}_V\otimes pf)\rho_{\mathcal F}=(\mathrm{id}_V\otimes qg)\rho_{\mathcal F} \overset{\eqref{E:Push1}}{=} (\mathrm{id}_V\otimes q)\rho_\f, 
\end{align*}
which we denote  by $\rho_K: V\to V\otimes K$. In view of \cref{P:UH(e)}(a) and \cref{P:UH(f)}(c), the sequence $\xymatrix{
\kk \f^*\ar@{^{(}->}[r]&V^{\otimes m} \ar[r]^-{\hat \e} &M_\e
}$
is  a sequence of $K$-comodule maps via $\rho_K$. By \cref{P:UHef}(e), there is a unique Hopf algebra map $\Psi: \cH(\e,\f)\to K$ such that $(\mathrm{id}_V\otimes \Psi)\rho_\cH=\rho_K$.  Hence with \eqref{E:Push2}, we get $(\mathrm{id}_V\otimes \Psi\alpha)\rho_\e=\rho_K=(\mathrm{id}_V\otimes p)\rho_\e$. So, $p=\Psi\alpha$ by the universal property of $\cH(\e)$. Similarly, we have $q=\Psi\beta$. The uniqueness of $\Psi$ is clear. This proves our result.
\end{proof}

\begin{remark}
Note that \Cref{push} both proves the existence of a Hopf algebra with the universal properties required of $\cH(\e,\f)$ and provides a presentation (e.g., generators and relations) for $\cH(\e,\f)$ simply by combining the presentations for $\cH(\e)$ and $\cH(\f)$. This is because, according to \cref{P:PushA}, \Cref{eq:push} is a pushout in the category of algebras as well as that of Hopf algebras. 
\end{remark}

Parallel to \Cref{push}, we have maps $f':\mathcal F\to \cS(\e)$, $g': \mathcal F\to \cS(\f)$ and $\alpha': \cS(\e)\to \cS(\e,\f)$, $\beta': \cS(\f)\to \cS(\e,\f)$ satisfying $\alpha' f'=\beta' g'$. The following result can be proved in the same manner as \Cref{push} using \cref{R:USL}.

\begin{theorem}\label{push_sl}
The following commutative diagram is the pushout of $\cS(\e)$ and $\cS(\f)$ along $\mathcal F$. 
\begin{align}\label{eq:push_sl}
\begin{tabular}{c}
\xymatrix{
 &  \cS(\e)\ar[dr]_-{\alpha'}\ar@/^1pc/[rrrd]^-{p'} &  &   &\\
\mathcal F\ar[ur]^-{f'}\ar[dr]_-{g'} &    & \cS(\e,\f)\ar@{-->}[rr]^\Psi && K\\
  &  \cS(\f)\ar[ru]^-{\beta'}\ar@/_1pc/[rrru]_-{q'} & &
}
\end{tabular}
\end{align}

\vspace{-.2in}

\qed
\end{theorem}

\vspace{.05in}
  
\subsection{Properties of $\cS(\e)$, $\cS(\f)$, $\cS(\e,\f)$ and $\cH( \e , \f )$}\label{sec:properties}
Now we study further properties of $\cH( \e , \f )$ and related Hopf algebras. We start by establishing a sufficient condition when the elements ${\sf D}_ \e $ and ${\sf D}_ \f^{-1}$ are equal in $\cH( \e ,  \f )$, and further, when these elements are central.

\begin{notation}[$ \e  \odot  \f $, $ \e  \star  \f $] \label{N:odot} Consider the following scalars:
$$ \e \odot  \f :=\textstyle \sum_{i_1, \dots, i_m=1}^n  \e _{i_1\cdots i_m} \f _{i_1\cdots i_m}
\quad \quad \text{and} \quad \quad ( \e \star \f )_{ij}=\sum_{i_1,\dots, i_{m-1}=1}^n  \e _{ii_1\cdots i_{m-1}} \f _{i_1\cdots i_{m-1}j}.$$
It is clear that $ \e \odot \f = \f \odot \e $.  Moreover, it is easy to check that $ \f \star  \e =\mathbb Q^{-T}({\sf e\star \sf f})^T\mathbb P$ by \Cref{rotate}.
\end{notation}

\begin{proposition}\label{P:PropHEF}
Let $K$ be a Hopf algebra that right coacts on  $V$, on $M_\e$ and on $\kk \f^*$ via \linebreak $\rho_K(v_{j})=\sum_{i=1}^n v_i\otimes x_{ij}$, $\rho_K(m_\e)=m_\e\otimes d$ and $\rho_K(\f^*)=\f^*\otimes d'$, resp., for some $(x_{ij})_{1\le i,j\le n}\in K$ and two grouplike elements $d,d'\in K$. Then, the following statements hold.
\begin{enumerate}
\item If $ \e \odot  \f \neq 0$, then $d=d'$ in $K$.
\item  $\mathbb X\, ( \f \star  \e )\, d=d'\, ( \f \star  \e )\, \mathbb X$ and $\mathbb X\, ( \e \star  \f )^T\, d'^{-1}=d^{-1} ( \e \star  \f )^T\,\mathbb X$, where $\mathbb X$ is the matrix  $(x_{ij})_{1\le i,j\le n}$.
\end{enumerate}
Further, suppose $ \e \odot  \f \neq 0$. 
\begin{enumerate}
\item[(c)] If $V$ is an inner-faithful $K$-comodule, and if either $\f\star \e$ or $\e\star \f$ is a nonzero scalar multiple of $\I$, then $d=d'$ is in the center of $K$.
\end{enumerate}  
\end{proposition}

\begin{proof}
(a) Denote by $\iota: \kk \f^*\hookrightarrow V^{\otimes m}$ the natural inclusion, ${\sf f}^* \mapsto \sum_{i_1,\dots,i_m=1}^n {\sf f}_{i_1 \cdots i_m} v_{i_1} \otimes \cdots \otimes v_{i_m}$. Observe that the composition of right $K$-comodule maps $\hat \e\circ \iota: \kk \f^*\hookrightarrow V^{\otimes m}\to M_\e$ is given by ${\sf f}^* \mapsto ({\sf e} \odot {\sf f}) m_{\e}$. Hence, if $ \e \odot  \f \neq 0$, we get $d'=d$ in $K$ by Schur's Lemma.

\smallskip

(b) It is straightforward to check that the following composition of $K$-comodule maps 
\[
\xymatrix{
\kk \f^*\otimes V\ar[rr]^-{ \iota\otimes \text{id}_V} && V^{\otimes (m+1)}\ar[rr]^-{\text{id}_V\otimes \hat \e } &&V\otimes M_\e,
}
\]
is induced by $( \f \star  \e )$ such that $\f^* \otimes v_i\mapsto \sum_{j=1}^n(\f \star  \e )_{ji}(v_{j}\otimes m_\e)$. By using the  matrix of coefficients $\mathbb X$, we get the first identity $\mathbb X( \f \star  \e )d=d'( \f \star  \e )\mathbb X$. Indeed,
\[
\begin{array}{ll}
\smallskip

\textstyle  \sum_{\ell=1}^n  v_\ell \otimes m_\e  \otimes \bigg(\sum_{k=1}^n d' ( \f  \star  \e )_{\ell k} x_{ki}\bigg)
&=\sum_{k,\ell=1}^n ( \f  \star  \e )_{\ell k} v_\ell \otimes m_\e  \otimes d' x_{ki}\\
\smallskip

=(( \f  \star  \e )\otimes \mathrm{id}_K)(\sum_{k=1}^n \f^*\otimes v_k \otimes d' x_{ki})
&=( ( \f  \star  \e )\otimes \mathrm{id}_K)(\rho_K (\f^* \otimes v_i))\\
\smallskip

=\rho_K(( \f  \star  \e )  (\f^* \otimes v_i))
&=\rho_K (\sum_{j=1}^n ( \f  \star  \e )_{ji} (v_j \otimes m_\e))\\
\smallskip

= \sum_{j,\ell=1}^n ( \f  \star  \e )_{ji} v_\ell \otimes m_\e \otimes x_{\ell j} d
&=\sum_{\ell=1}^n  v_\ell \otimes m_\e  \otimes \bigg(\sum_{j =1}^n  x_{\ell j} ( \f  \star  \e )_{ji} d\bigg).
\end{array}
\]
The second identity can be obtained similarly by considering the map from $V\otimes \kk \f^*$ to $M_\e\otimes V$. 

\smallskip

(c) By part (a), we know $d=d'$. Moreover, $\f\star \e=\I$ or $\e\star \f=\I$ implies that $\mathbb Xd=d\mathbb X$ by part~(b). So $d=d'$ commutes with $x_{ij}$'s. Now by inner-faithfulness, $K$ is generated by $\mathbb X$. Thus, $d=d'$ is in the center of $K$. 
\end{proof}

\begin{corollary} \label{C:central}
The following statements hold for $\cH(\e,\f)$.
\begin{enumerate}
\item If $ \e \odot  \f \neq 0$, then ${\sf D}_\e{\sf D}_\f=1$. 
\item $\mathbb U\, ( \f \star  \e )\, {\sf D}_\e={\sf D}_\f^{-1}\, ( \f \star  \e )\, \mathbb U$ and $\mathbb U\, ( \e \star  \f )^T\, {\sf D}_\f={\sf D}_\e^{-1}\, ( \e \star  \f )^T\,\mathbb U$, where $\mathbb U$ is the matrix  $(u_{ij})_{1\le i,j\le n}$.
\end{enumerate}
Further, suppose $ \e \odot  \f \neq 0$.
\begin{enumerate}
\item[(c)] If either $\f\star \e$ or $\e\star \f$ is  a nonzero scalar multiple of $\I$, then ${\sf D}_\e={\sf D}_\f^{-1}$ is in the center of $\cH(\e,\f)$.
\end{enumerate}
\end{corollary}

\begin{proof}
We can apply $K= \cH( \e ,  \f )$ in the result above with $d = {\sf D}_ \e $ and $d' = {\sf D}_ \f^{-1}$ according to \cref{P:UHef}(d). 
\end{proof}

Now we discuss the cosovereignty property of the quantum special linear groups in this work. Recall that for any algebra $A$ and coalgebra $C$, the convolution product in $\mathrm{Hom}_\kk(C,A)$ is given by $f*g=m_A(f\otimes g)\Delta_C$ for any $f,g\in \mathrm{Hom}_\kk(C,A)$ with unit $u_A\varepsilon_C$.

\begin{definition}[cosovereign Hopf algebra, sovereign character] \cite[Definition 2.7]{cosv}\label{def.cosv}
An algebra morphism $\Phi: H\to \kk$ is said to be a {\it sovereign character} on $H$ if $S^2 = \Phi \ast \text{Id} \ast \Phi^{-1}$. A {\it cosovereign Hopf algebra} is a Hopf algebra equipped with a sovereign character. 
\end{definition}

We turn our attention to the cosovereign property being preserved under pushout. The following lemma is straightforward. 

\begin{lemma}\label{L:Cosov}
Let $H$ be any Hopf algebra, and $\Phi: H\to \kk$ be an algebra map. Then, $\Phi^{-1}=\Phi\circ S$.

Further, suppose $S^2=\Phi* \mathrm{id}*\Phi^{-1}$ holds for some $a_i$s in $H$. If $\{a_i, S(a_i),S^2(a_i),\dots\}$ generate $H$ as an algebra, then $S^2=\Phi* \mathrm{id}*\Phi^{-1}$ for $H$.  \qed
\end{lemma} 

\begin{proposition}\label{P:PushCov}
Let $A,B,C$ be three Hopf algebras with Hopf algebra maps $f: C\to A$ and $g: C\to B$. Let $A\amalg_C B$ be the pushout of $A$ and $B$ along $C$ with Hopf algebra maps $\alpha: A\to A\amalg_C B$ and $\beta: B\to A\amalg_C B$ satisfying $\alpha f=\beta g$. Suppose $A$ and $B$ are cosovereign with sovereign characters $\Phi_A: A\to \kk$ and $\Phi_B: B\to \kk$, respectively. If $\Phi_A f=\Phi_B g$, then $A\amalg_C B$ is cosovereign with sovereign character $\Phi: A\amalg_C B\to \kk$ uniquely determined by $\Phi_A=\Phi\alpha$ and $\Phi_B=\Phi \beta$.
\end{proposition}

\begin{proof}
In view of \Cref{P:PushA}, $A\amalg_CB$ is also the pushout of $A$ and $B$ along $C$ in the category of algebras. By its universal property, there is a unique algebra map $\Phi: A\amalg_C B\to \kk$ such that $\Phi_A=\Phi\alpha$ and $\Phi_B=\Phi \beta$. It remains to show that $\Phi$ is a sovereign character on $\cH(\e,\f)$. 

For any $a\in A$, we have
\[
\begin{array}{rll}
S^2(\alpha(a))& =\alpha(S^2(a))=\alpha(\ts \sum_{(a)}\Phi_A(a_1)\, a_2\, \Phi^{-1}_A(a_3)) &=\sum_{(a)} \Phi_A(a_1)\, \alpha(a_2)\, \Phi_A(S(a_3))\\
&=\ts \sum_{(a)} \Phi(\alpha(a_1))\, \alpha(a_2)\, \Phi(\alpha(S(a_3)))&=\ts \sum_{(a)} \Phi(\alpha(a_1))\, \alpha(a_2)\, (\Phi\circ S)(\alpha(a_3))\\
&= (\Phi*\mathrm{id}*\Phi^{-1})(\alpha(a)).
\end{array}
\]

\noindent Here, we use the fact that $\alpha: A\to A\amalg_CB$ is a Hopf algebra map along with \Cref{L:Cosov}. Similarly, we can show $S^2=\Phi*\mathrm{id}*\Phi^{-1}$ holds for $\beta(b)$ for any $b\in B$. Now being a pushout of $A$ and $B$ along $C$ in the category of algebras implies that $A\amalg_CB$ is generated by $\alpha(A)$ and $\beta(B)$ as an algebra. By \Cref{L:Cosov}, $S^2=\Phi*\mathrm{id}*\Phi^{-1}$ holds for $A\amalg_CB$. This proves our result.  
\end{proof}

\begin{theorem}\label{pr.cosv_SeSf}
The following universal quantum groups are cosovereign:
\begin{enumerate}
\item $\cS(\e)$ with sovereign character $\Phi_\e: \cS(\e)\to \kk$ given by $\Phi_\e(\bA)=\bP^{-1}$ and $\Phi_\e(\bB)=\bP$;
\smallskip

\item $\cS(\f)$  with sovereign character $\Phi_\f: \cS(\f)\to \kk$ given by $\Phi_\f(\bA)=\bQ^{-T}$ and $\Phi_\f(\bB)=\bQ^T$; and 
\smallskip

\item $\cS( \e , \f )$ with sovereign character $\Phi: \cS(\e,\f)\to \kk$ given by $\Phi(\mathbb U)=\bP^{-1}$, when  $\bP=\bQ^T$. 
 \end{enumerate}
\end{theorem}

\begin{proof}
(a) It follows from the relations \Cref{He-aij,He-bij} (with ${\sf D}_\e=1$) and \Cref{rotate} that $\Phi_\e$ extends to an algebra morphism $\Phi_\e:\cS(\e)\to \kk$. By \Cref{C:S2He} with ${\sf D}_\e=1$, we have that $S^2=\eta_\e$ in $\cS(\e)$. Hence, one sees that 
\[
\begin{array}{rll}
S^2(a_{ij})&=\ts \sum_{k,\ell=1}^n \mathbb P^{-1}_{ik}\, a_{k\ell}\, \mathbb P_{\ell j}&=\ts \sum_{k,\ell=1}^n \Phi_\e(a_{ik})\, a_{k\ell}\, \Phi_\e(b_{\ell j})\\
&=\ts \sum_{k,\ell=1}^n \Phi_\e(a_{ik})\, a_{k\ell}\, \Phi_\e(S(a_{\ell j}))&=(\Phi_\e*\mathrm{id}*{\Phi_\e}^{-1})(a_{ij}),
\end{array}
\]
where we use Lemma~\ref{L:Cosov} for the last equality. Again by \Cref{L:Cosov}, $\cS(\e)$ is cosovereign. The argument is analogous for part (b). 

\smallskip

(c) Recall the notation of Section~\ref{sec:pushout}. The algebra map $\Phi_\e:\cS(\e)\to \kk$ of part (a) gives an algebra map $\Phi_\e f': \mathcal F\to \kk$ with $t_{ij} \mapsto \mathbb P^{-1}_{ij}$. Similarly, the algebra map $\Phi_\f:\cS(\f)\to \kk$ gives rise to an algebra map $\Phi_\f g': \mathcal F\to \kk$ with $t_{ij}\mapsto \mathbb Q^{-T}_{ij}$ (note that  $\phi\mapsto (\phi(t_{ij}))_{1\le i,j\le n}$ yields a bijection between $\mathrm{Hom}_{\mathrm{Alg}}(\mathcal F,\kk)$ and $\mathrm{GL}_n(\kk)$). Under the hypothesis $\bP=\bQ^T$, we get $\Phi_\e f'=\Phi_\f g'$. Then, the result follows from \Cref{P:PushCov} and parts (a) and (b).
\end{proof}



\section{Examples of $\mathcal H( \e , \f )$ arising in the literature} \label{oldexamples}

In this section, we examine several special cases of the Hopf algebra $\cH( \e ,  \f )$ that have appeared in other articles.

\subsection{Mrozinski's  and Dubois-Violette and Launer's quantum groups}

Note that when $m=2$, we get that $ \e ,  \f $ are bilinear forms identified with matrices $\mathbb{E}=(e_{ij}), \mathbb{F}=(f_{ij}) \in GL_n(\kk)$, respectively. Namely, $e_{ij} =  \e (v_i \otimes v_j)$ and $f_{ij} =  \f (\theta_i \otimes \theta_j)$ for $1 \leq i,j \leq n$. Moreover, $\widetilde \e $, $\widetilde \f $ are unique and correspond to $\mathbb{E}^{-1}, \mathbb{F}^{-1}$, respectively. We obtain the following Hopf algebra in this case.

\begin{definition}[$\cH(\bE, \bF)$] The Hopf algebra $\cH(\bE, \bF)$ is defined as the Hopf algebra $\cH( \e ,  \f )$ with $m=2$, where the matrices $\bE, \bF \in GL_n(\kk)$ correspond to the forms $ \e ,  \f $. In other words, $\cH(\bE, \bF)$ is generated by $\bU:= (u_{ij})_{1 \leq i,j \leq n}$, ${\sf D_e}^{\pm 1}$, ${\sf D_f}^{\pm 1}$ with relations 
$$ \bE^{-1} \bU^T \bE \bU = {\sf D_e} \bI, \quad \bU \bF \bU^T \bF^{-1} = {\sf D_f}^{-1} \bI, \quad 
{\sf D_e}{\sf D_e}^{-1} = {\sf D_e}^{- 1} {\sf D_e} = {\sf D_f}{\sf D_f}^{-1} = {\sf D_f}^{- 1} {\sf D_f} = 1_{\mathcal{H}},$$
with $\Delta(u_{ij}) = \sum_{k=1}^n u_{ik} \otimes u_{kj}$, $\varepsilon(u_{ij}) = \delta_{ij}$, and $S(\bU) = {\sf D_e}^{-1} \bE^{-1} \mathbb U^T \bE = \bF\bU^T \bF^{-1} {\sf D_f}$.
\end{definition} 
 
Now, Mrozinski's $GL(2)$-like quantum group $\cG(\bE, \bF)$ (\cite{Mrozinski}, \cite[Definition~2.10(b)]{Wsquared})  and Dubois-Violette and Launer's quantum group $\cB(\bE)$ (\cite{DVL}, \cite[Definition~2.10(a)]{Wsquared})  both arise as Hopf quotients of $\cH(\bE, \bF)$ as follows.
 
 \begin{proposition}
 We obtain that 
 $$\cG(\bE, \bF) \cong \cH(\bE, \bF)/({\sf D_e}{\sf D_f} - 1_{\cH}) \quad
 \text{ and } \quad \cB(\bE) \cong \cS(\bE, \bE^{-1}),$$
  as Hopf algebras. \qed
 \end{proposition}

\begin{definition}[$\cG(\e, \f)$]\label{D:G(e,f)} 
Generally, we define $\cG(\e,\f)=\cH(\e,\f)/({\sf D}_\e{\sf D}_\f-1)$ to be the quotient Hopf algebra of $\cH(\e,\f)$ by identifying ${\sf D}_\e$ with ${\sf D}_\f^{-1}$. 
\end{definition}

\subsection{Artin-Schelter-Tate's quantum deformation of $\mathcal{O}(GL_n)$}  \label{AST}
Artin, Schelter and Tate studied  in \cite{AST} a quantum deformation of $\mathcal{O}(GL_n)$ by considering the universal quantum group coacting on a pair of skew polynomial rings: 
\begin{align*}
\kk_{{\bf q}}[x_1, \dots, x_n]:&=\kk \langle x_1, \dots, x_n \rangle/(x_j x_i - q_{ji} x_i x_j)\ \text{and}\\
\kk_{{\bf p}}[y_1, \dots, y_n]:&=\kk \langle y_1, \dots, y_n \rangle/(y_j y_i - p_{ij} y_i y_j).
\end{align*}
Here ${\bf p}=(p_{ij})$ and ${\bf q}=(q_{ji})$ are multiplicative anti-symmetric matrices in $M_n(\kk^{\times})$, that is
\begin{align}\label{E:RRAST}
p_{ij}p_{ji}=1,\ p_{ii}=1, \ q_{ji}q_{ij}=1,\ q_{ii}=1 \quad \text{ for } i \neq j.
\end{align}
It is well-known that these skew polynomial rings are Koszul AS-regular algebras of global dimension~$n$.

Next, let $V = \bigoplus_{k=1}^n \kk y_k$ be the generating space of $\kk_{\bf p}[y_1, \dots, y_n]$, and $V^* = \bigoplus_{k=1}^n \kk x_k$ be the generating space of $\kk_{\bf q}[x_1, \dots, x_n]$ with $x_k = y_k^*$. 
Considering \cref{propA}, we define two preregular $n$-linear  forms $ \e _{\bf q}: V^{\otimes n}\to \kk$ and  $ \f _{\bf p}: (V^*)^{\otimes n}\to \kk$ associated with ${\bf q}$ and ${\bf p}$ by

\begin{align}\label{E:QAST}
{\small  \e _{\bf q}(y_{i_1},\dots,y_{i_n})=
\begin{cases}
\prod\limits_{\substack{j<j'\\ i_j>i_{j'}}}\left(-q_{i_{j'} i_{j}}\right)  & \text{if}\ (i_1,\dots,i_n)=(\sigma(1),\dots,\sigma(n))\ \text{for some}\ \sigma\in S_n\\
0 & \text{otherwise}
\end{cases}}
\end{align}
\begin{align}\label{E:PAST}
{\small \f _{\bf p}(x_{i_1},\dots,x_{i_n})=
\begin{cases}
\prod\limits_{\substack{j<j'\\ i_j>i_{j'}}}\left(-p_{ i_{j} i_{j'}}\right)  & \text{if}\ (i_1,\dots,i_n)=(\sigma(1),\dots,\sigma(n))\ \text{for some}\ \sigma\in S_n\\
0 & \text{otherwise.}
\end{cases}}
\end{align}

\begin{lemma}\label{L:FAST}
We have that $\kk_{\bf q}[x_1, \dots, x_n]\cong A( \e _{\bf q},2)$ and $\kk_{\bf p}[y_1, \dots, y_n]\cong A( \f _{\bf p},2)$.
\end{lemma}

\begin{proof}
We establish the first isomorphism here; the second one follows from the same argument. We consider its Koszul dual of $\kk_{\bf q}[x_1, \dots, x_n]$,  denoted by $S=\bigoplus_{i\geq 0} S_i$. To avoid confusion, suppose that the dual space $V^{**}\cong V$ of $V^*$ has dual basis $\{\xi_1,\xi_2,\dots,\xi_n\}$. Hence, $S$ is generated by degree one elements $\{\xi_1,\xi_2,\dots,\xi_n\}$  with the defining relations 
\[
\xi_i^2=0,\quad \xi_j\xi_i=-q_{ij}\xi_i\xi_j,\quad \quad \text{for } 1\le i,j\le n;
\]
see \cite[page 881]{AST}.
It is clear that $S_i=0$ for $i>n$ and $\dim S_n=1$. Now, \Cref{N-Koszul}(e) implies that $\kk_{\bf q}[x_1, \dots, x_n]$ is isomorphic to a superpotential algebra $A({\sf s},2)$, where ${\sf s}: V^{\otimes n}\to \kk$ is induced by the multiplication in $S$, namely, $V^{\otimes n}=S_1^{\otimes n}\twoheadrightarrow S_n\cong \kk$. We choose a basis $\{\xi_1\xi_2\cdots \xi_n\}$ of $S_n$ and we obtain
\[
{\small
{\sf s}(\xi_{i_1}\xi_{i_2}\cdots \xi_{i_n})=
\begin{cases}
\prod\limits_{\substack{j<j'\\ i_j>i_{j'}}}\left(-q_{i_{j'} i_{j}}\right)  & \text{if}\ (i_1,\cdots,i_n)=(\sigma(1),\dots,\sigma(n))\ \text{for some}\ \sigma\in S_n\\
0 & \text{otherwise.}
\end{cases}
}
\]
by extracting the coefficient needed to write $\xi_{i_1}\xi_{i_2}\cdots \xi_{i_n}$ in terms of $\xi_1\xi_2\cdots \xi_n$. Therefore, we get  ${\sf s} =  \e _{\bf q}$ as described in \eqref{E:QAST}. 
\end{proof}

\begin{hypothesis}[$\lambda$] \label{hypAST}
For the rest of this part, assume that for $A( \e _{\bf q},2)$ and $A( \f _{\bf p},2)$ there is a scalar $\lambda \neq -1$ in $\kk$ so that $q_{ji}=\lambda p_{ji}$ for all $j>i$, 
\end{hypothesis}

Now we examine Hopf algebras that have a balanced coaction  on $A( \e _{\bf q},2)$ and $A( \f _{\bf p},2)$.

\begin{lemma}\label{L:RAST}
 Let $H$ be any Hopf algebra that has a balanced coaction on $A( \e _{\bf q},2)$ and $ A( \f _{\bf p},2)$  via
\begin{equation}\label{ASTcoact}
\textstyle x_j\mapsto \sum_{i=1}^n u_{ji}\otimes x_i  \quad \text{ and } \quad y_j\mapsto \sum_{i=1}^n y_i\otimes u_{ij}
\end{equation}
for some $u_{ij} \in H$. Then, the following statements hold.
\begin{enumerate}
\item We obtain the relations of $H$ below
\begin{align}\label{E:URAST}
u_{j\beta}\,u_{i\alpha}=
\begin{cases}
p_{ji}\,p_{\alpha\beta}\, u_{i\alpha} \,u_{j\beta}+(\lambda-1)\,p_{ji}\,u_{i\beta}\,u_{j\alpha}, & \text{if}\ j>i,~~\beta>\alpha\\
\lambda \,p_{ji}\,p_{\alpha\beta}\, u_{j\alpha}\,u_{j\beta},  & \text{if}\ j>i,~~\beta \le \alpha\\
p_{\alpha\beta}\, u_{i\alpha}\,u_{j\beta}, & \text{if}\ j=i, ~~\beta>\alpha. 
\end{cases}
\end{align}
\item The element ${\sf D}$ defined by the following equality 
\begin{align}\label{E:ASTD}
\quad \quad \quad \textstyle \sum_{i_1,i_2,\dots, i_n=1}^n ~~( \e _{\bf q})_{i_1 \cdots i_n}\,u_{i_11}\,u_{i_22}\cdots u_{i_nn}~~=~~\sum_{i_1,i_2,\dots, i_n=1}^n ~~( \f _{\bf p})_{i_1\cdots i_n}\,u_{1i_1}\,u_{2i_2}\cdots u_{ni_n}
\end{align}
is a grouplike element in $H$. 
\item The homological codeterminant for the left $H$-coaction on $A( \e _{\bf q},2)$ is ${\sf D}$, and the homological codeterminant for the right $H$-coaction on $A( \f _{\bf p},2)$ is  ${\sf D}^{-1}$.
\end{enumerate}
\end{lemma}

\begin{proof}
Part (a) follows from \cite[Theorem 1]{AST}. The equality in part (b) can be derived from the relations in part (a) or see \cite[Lemma 1]{AST}. The fact that ${\sf D}$ is a grouplike element and part (c) follow from \Cref{T:CoactionA(e)}((a)$\Rightarrow$(e)) and \Cref{T:CoactionA(f)}((a)$\Rightarrow$(e)) noticing that $(\e _{\bf q})_{1 \cdots n}=(\f_{\bf p})_{1 \cdots n}=1$.
\end{proof}

\begin{proposition}\cite{AST} \label{P:BAST}
The following statements hold for the bialgebra $\mathcal O_{A( \e _{\bf q},2), A(\f_{\bf p},2)}(M)=:\mathcal B$ associated to $A( \e _{\bf q},2)$ and $A( \f _{\bf p},2)$.
\begin{enumerate}
\item $\mathcal B$ has generators $(u_{ij})_{1\le i,j\le n}$ defined by the relations \eqref{E:URAST} with $\Delta(u_{ij})=\sum_{k=1}^n u_{ik}\otimes u_{kj}$ and $\varepsilon(u_{ij})=\delta_{ij}$. 
\item $\mathcal B$ is a right and left Noetherian domain and Koszul Artin-Schelter regular of dimension $n^2$.
\item The element ${\sf D}$ defined in \eqref{E:ASTD} is a regular normal element of $\mathcal B$. Hence, the multiplicative set $\langle {\sf D}\rangle$ is an Ore set of $\mathcal B$. 
\end{enumerate}
\end{proposition}
\begin{proof}
By \cite[Theorem 2]{AST}, $\mathcal B$ is a well-defined. Hence part (a) follows from the universal property of $\mathcal O_{A( \e _{\bf q},2), A(\f_{\bf p},2)}(M)$ stated in \Cref{P:QMSHef} and \Cref{L:RAST}(a). Part (b) is also follows from \cite[Theorem 2]{AST} and part (c) follows from \cite[Theorem 3]{AST}.
\end{proof}

Next, we recall the presentation of Artin-Schelter-Tate's  quantum group defined in \cite{AST}.

\begin{definition}[$\cO_{{\bf p},\lambda}(GL_n)$] \cite[Theorem~3]{AST}\label{D:HAST}
Take $\cO_{{\bf p},\lambda}(GL_n)$ to be the Hopf algebra generated by $\bU:=(u_{ij})_{1 \leq i,j \leq n}$ and grouplike elements ${\sf D}^{\pm 1}$ satisfying the relations of \cref{L:RAST}(a,b), with 
$\Delta(u_{ij})=\sum_{k=1}^n u_{ik}\otimes u_{kj}$ and $\varepsilon(u_{ij})=\delta_{ij}$ and with antipode given by
\[
\begin{array}{rl}
\bigskip
S(u_{jk})&=\frac{\textstyle \prod_{m=k+1}^{n}\left(-q_{km}\right)}{\textstyle \prod_{m=j+1}^{n}\left(-q_{jm}\right)} \,{\sf D}^{-1}\,\bigg(\sum_{f\in [\hat{j},\hat{k}]}\sigma(p,f)\prod_{i\in \hat{j}} u_{i,f(i)}\bigg)\\

&=\frac{\textstyle\prod_{m=1}^{j-1}\left(-p_{jm}\right)}{\textstyle\prod_{m=1}^{k-1}\left(-p_{km}\right)}\,\bigg(\sum_{g\in [\hat{k},\hat{j}]}\sigma(q^{-1},g)\prod_{i\in \hat{k}} u_{g(i),i}\bigg) \,{\sf D}^{-1},
\end{array}
\]
where $\hat{\ell}:=\{1,2,\dots,n\}\setminus\{\ell\}$, and $[\hat{\ell},\hat{\ell'}]$ denotes the set of bijections from  $\hat{\ell}$ to $\hat{\ell'}$, and \linebreak $\sigma(r,h):=\prod_{j<j',h(j)>h({j'})} (-r_{h(j),h({j'})})$ for $r =p,q$ and $h=f,g$.
\end{definition}

Finally, we present the main result of this section.

\begin{proposition}\label{P:AST}
Let $ \e _{\bf q}$ and $ \f _{\bf p}$ be two preregular $n$-linear forms defined in \eqref{E:QAST} and \eqref{E:PAST}. Under \Cref{hypAST},  we have the following isomorphisms of universal quantum groups
\[
\mathcal O_{{\bf p},\lambda}(GL_n)~ \cong ~\mathcal O_{A( \e _{\bf q},2), A(\f_{\bf p},2)}(M){\sf D}^{-1} ~ \cong ~ \mathcal O_{\kk_{\bf q}[x_1,\dots,x_n],\kk_{\bf p}[y_1,\dots,y_n]}(GL) ~\cong ~  \mathcal H( \e _{\bf q}, \f _{\bf p}) ~ \cong ~ \mathcal G( \e _{\bf q}, \f _{\bf p}).
\]
\end{proposition}

\begin{proof}
The first isomorphism follows from \Cref{P:BAST}(a) and \Cref{D:HAST}. By \Cref{P:BAST}(c), the multiplicative set $\langle {\sf D}\rangle$ is an Ore set of $\cO_{A(\e _{\bf q},2), A(\f_{\bf p},2)}(M)$. Hence the second isomorphism is obtained by \Cref{T:LocalBi}(a) and \Cref{L:FAST}. 
The third isomorphism is \Cref{H-Oisom}(a). The last isomorphism follows from \cref{L:RAST}(b).
\end{proof}

\subsection{Takeuchi's two-parameter quantum deformation of $\mathcal{O}(GL_n)$} \label{Takeuchi}
Takeuchi in \cite{Takeuchi} constructed a two-parameter quantum deformation $\cO_{p,q}(GL_n)$ of $\cO(GL_n)$ depending on two units $p,q \in \kk$. As pointed out in \cite{AST}, this family can be obtained as a special case of the Hopf algebras $\mathcal O_{{\bf p},\lambda}(GL_n)$ discussed in the previous section. To define $\cO_{p,q}(GL_n)$, define two multiplicatively anti-symmetric $n\times n$-matrices ${\bf p}=(p_{ij})$ and ${\bf q}=(q_{ji})$ according to \eqref{E:RRAST} such that 
\begin{align*}
p_{ij}=p \quad \text{and} \quad q_{ji}=q\ \text{for all}\ j>i.
\end{align*}
We also define two skew polynomial rings 
\begin{align*}
\kk_{q}[x_1, \dots, x_n]:&=\kk \langle x_1, \dots, x_n \rangle/(x_j x_i - q x_i x_j,\ j>i)\ \text{and}\\
\kk_{p}[y_1, \dots, y_n]:&=\kk \langle y_1, \dots, y_n \rangle/(y_j y_i - p y_i y_j,\ j>i).
\end{align*}

Next, let $V = \bigoplus_{k=1}^n \kk y_k$ be the generating space of $\kk_{p}[y_1, \dots, y_n]$, and $V^* = \bigoplus_{k=1}^n \kk x_k$ be the generating space of $\kk_{q}[x_1, \dots, x_n]$ with $x_k = y_k^*$. 
Following \eqref{E:QAST} and \eqref{E:PAST}, the two preregular $n$-linear  forms $ \e _q: V^{\otimes n}\to \kk$ and  $ \f _p: V^{*\otimes n}\to \kk$ associated with $q$ and $p$ are given by 
\begin{align}\label{E:QTake}
 \e _q(y_{i_1},\cdots,y_{i_n})=
\begin{cases}
(-q)^{-\ell(\sigma)} & \text{if}\ (i_1,\cdots,i_n)=(\sigma(1),\cdots,\sigma(n))\ \text{for some}\ \sigma\in S_n\\
0 & \text{otherwise}
\end{cases}
\end{align}
\begin{align}\label{E:PTake}
 \f _p(x_{i_1},\cdots,x_{i_n})=
\begin{cases}
(-p)^{-\ell(\sigma)} & \text{if}\ (i_1,\cdots,i_n)=(\sigma(1),\cdots,\sigma(n))\ \text{for some}\ \sigma\in S_n\\
0 & \text{otherwise},
\end{cases}
\end{align}
where $\ell(\sigma)$ denotes the number of inversions for a permutation $\sigma\in S_n$. We obtain from  \cref{L:FAST}  that $\kk_{q}[x_1, \dots, x_n]\cong A( \e _q,2)$ and $\kk_{p}[y_1, \dots, y_n]\cong A( \f _p,2)$.

Now we examine Hopf algebras that have a balanced coaction on $A( \e _q,2)$ and $A( \f _p,2)$. The  result below follows from \cref{L:RAST}.

\begin{lemma}\label{L:RTake}
Let $H$ be any Hopf algebra that left coacts on $A( \e _q,2)$ and right coacts on $ A( \f _p,2)$  via 
$x_j\mapsto \sum_{i=1}^n u_{ji}\otimes x_i$ and $y_j\mapsto \sum_{i=1}^n y_i\otimes u_{ij}$ for some $u_{ij} \in H$. 
If $pq\neq -1$, then we have the following relations in $H$
\begin{align*}\label{E:URAST}
u_{j\beta}\, u_{i\alpha}=
\begin{cases}
u_{i\alpha}\, u_{j\beta}+(q-p^{-1})\, u_{i\beta}\, u_{j\alpha}, & \text{if}\ j>i,~~\beta>\alpha\\
q\, u_{j\alpha}\, u_{j\beta},  & \text{if}\ j>i,~~\beta=\alpha\\
qp^{-1}\, u_{j\alpha}\, u_{j\beta},  & \text{if}\ j>i, ~~\beta < \alpha\\
p\, u_{i\alpha}\, u_{j\beta}, & \text{if}\ j=i, ~~\beta>\alpha. 
\end{cases}
\end{align*}
Moreover, the element 
\[
{\sf D}:=~~\textstyle \sum_{\sigma\in S_n}(-q)^{-\ell(\sigma)}u_{\sigma(1)1}u_{\sigma(2)2}\cdots u_{\sigma(n)n}~~=~~\sum_{\sigma\in S_n}(-p)^{-\ell(\sigma)}u_{1\sigma(1)}u_{2\sigma(2)}\cdots u_{n\sigma(n)}
\]
is a grouplike element in $H$. Further, the homological codeterminant for the left $H$-coaction on $A( \e _q,2)$ is given by ${\sf D}$ and the homological codeterminant for the right $H$-coaction on $A( \f _p,2)$ is given by ${\sf D}^{-1}$. \qed
\end{lemma}

%
%

Next, we recall the presentation of the quantum group defined by Takeuchi in \cite{AST}.
\begin{definition}[$\cO_{p,q}(GL_n)$] \cite{Takeuchi}
Take units $p,q \in \kk$. Let $\cO_{p,q}(GL_n)$ to be the Hopf algebra generated by $\bU:=(u_{ij})_{1 \leq i,j \leq n}$ and grouplike elements ${\sf D}^{\pm 1}$ satisfying the relations in \cref{L:RTake}, with $\Delta(u_{ij})=\sum_{k=1}^n u_{ik}\otimes u_{kj}$, $\varepsilon(u_{ij})=\delta_{ij}$ and with antipode given by 
\begin{align*}
S(u_{ij})~=~(-q)^{j-i}\,{\sf D}^{-1}\,|\mathbb U_{ji}|~=~(-p)^{j-i}\,|\mathbb U_{ji}|\,{\sf D}^{-1},
\end{align*}
where $|\mathbb U_{ji}|$ denotes the quantum determinant of the $(n-1)\times(n-1)$ minor obtained by removing the $j$-th row and the $i$-th column.
\end{definition}

By Proposition \ref{P:AST}, we obtain the result below.
\begin{proposition}\label{P:Take}
Let $p,q$ be two units in $\kk$ such that $pq\neq -1$. Suppose $ \e _q$ and $ \f _p$ are the two preregular $n$-linear forms defined in \eqref{E:QTake} and \eqref{E:PTake}. Then, we have the following isomorphisms of universal quantum groups. 
\[
\cO_{p,q}(GL_n)~\cong~\mathcal O_{A( \e _q,2), A(\f_p,2)}(M){\sf D}^{-1}~\cong ~\mathcal O_{\kk_q[x_1,\dots,x_n],\kk_p[y_1,\dots,y_n]}(GL)~\cong ~\mathcal H( \e _q, \f _p)~\cong ~\mathcal G( \e _q, \f _p).
\]
\qed
\end{proposition}

\begin{remark} \label{polynomial}
In particular, when $q=1$,  the associated preregular form $ \e := \e _1:V^{\otimes n}\to \kk$ is exactly the ``signature form'' $\varepsilon \in V^{* \otimes n}$ given in \cite[Section~7.1]{BDV}. Here, $\varepsilon_{i_1\cdots i_n}=0$  if two indices are equal, and $\varepsilon_{i_1\cdots i_n}$ is the signature of the corresponding permutation otherwise. The same is true for $ \f _1: V^{*\otimes n}\to \kk$.  So, by Proposition \ref{P:Take}, we get the isomorphisms of Hopf algebras below
\[
\mathcal O(GL_n)~\cong~\mathcal O_{A(\varepsilon,2), A(\varepsilon,2)}(M){\sf D}^{-1}~\cong ~\mathcal O_{\kk[x_1,\dots,x_n],\kk [y_1,\dots,y_n]}(GL)~\cong ~\mathcal H(\varepsilon,\varepsilon)~\cong~ \mathcal G(\varepsilon,\varepsilon).\]
\end{remark}


\section{New  quantum groups coacting on higher dimensional AS-regular algebras} \label{sec:examples}

Using our construction of the universal quantum group $\cH( \e , \f )$ from \cref{sec:Hef}, we provide new examples of Hopf algebras that coact universally on a pair of $N$-Koszul Artin-Schelter regular algebras of global dimensions 3 and 4. Here, we take $\kk$ to be a field of characteristic 0.

\subsection{On 3-dimensional Sklyanin algebras} \label{Skly3}
In \cite{Sklyanin}, Sklyanin considered certain graded algebras (now called {\it 4-dimensional Sklyanin algebras}) arising from the study of Yang-Baxter matrices and the related {\it Quantum Inverse Scattering Method}. We consider the 3-dimensional version of these algebras here, and postpone the study of 4-dimensional Sklyanin algebras for the next section.  

\begin{definition}[$Skl_3(a,b,c)$]\cite{AS,ATV1} The {\it 3-dimensional Sklyanin algebras}, denoted by $Skl_3(a,b,c)$, are the graded algebras  generated by variables $x,y,z$ of degree one, subject to relations
\begin{align}\label{E:RSkly3}
ayz+bzy+cx^2~=~azx+bxz+cy^2~=~
axy+byx+cz^2~=~0
\end{align}
for $[a:b:c]\in \mathbb P^2$ with $abc \neq 0$ and $(3abc)^3 \neq (a^3+b^3+c^3)^3$.
\end{definition}

The algebras $Skly_3(a,b,c)$ are Koszul and AS-regular of global dimension~$3$ by \cite{AS, ATV1}, and a result of J. Zhang \cite[Theorem~5.11]{Smith}. Let $V = \kk x \oplus \kk y \oplus \kk z$ be the generating space of $Skly_3(a,b,c)$ with dual space $V^*$ spanned by the dual basis $\{x^*,y^*,z^*\}$. Recall the notation of Example~\ref{exSabc}. By Example \ref{exSabc}(b), we have the isomorphism of algebras
$$Skly_3(a,b,c) ~\cong~ A({\sf s}_{abc},2).$$

Now we have that the universal quantum group that has a balanced coaction on a pair of 3-dimensional Sklyanin algebras $Skly_3(a',b',c')$ and $Skly_3(a,b,c)$ is understood by the following proposition; this holds by Theorem \ref{H-Oisom}.

\begin{proposition}
We have $\mathcal {O}_{Skly_3(a',b',c'),Skly_3(a,b,c)}(GL) \cong \mathcal H({\sf s}_{a'b'c'},{\sf s}_{abc})$ as Hopf algebras. \qed
\end{proposition}

\subsection{On 4-dimensional Sklyanin algebras}
Here,  we study the universal quantum group that has balanced coaction on a pair of 4-dimensional Sklyanin algebras. We follow the notation of \cite{SmithStafford}. 

\begin{definition}[$Skly_4(\alpha,\beta,\gamma)$] Let $\alpha,\beta,\gamma\in \kk$ satisfy
$\alpha+\beta+\gamma+\alpha\beta\gamma=0$, or equivalently
$$(1+\alpha)(1+\beta)(1+\gamma)=(1-\alpha)(1-\beta)(1-\gamma).$$
The {\it 4-dimensional Sklyanin algebras} $Skly_4(\alpha,\beta,\gamma)$ are the graded algebras generated by variables $x_0,x_1,x_2,x_3$ of degree one, subject to relations $f_i=0$, where
\begin{align*}
f_1:=x_0x_1-x_1x_0-\alpha(x_2x_3+x_3x_2),\quad f_2:=x_0x_1+x_1x_0-(x_2x_3-x_3x_2),\notag\\
f_3:=x_0x_2-x_2x_0-\beta(x_3x_1+x_1x_3),\quad f_4:=x_0x_2+x_2x_0-(x_3x_1-x_1x_3),\notag\\
f_5:=x_0x_3-x_3x_0-\gamma(x_1x_2+x_2x_1),\quad f_6:=x_0x_3+x_3x_0-(x_1x_2-x_2x_1).
\end{align*}
\end{definition}

By \cite[Theorem 0.3]{SmithStafford}, $Skly_4(\alpha,\beta,\gamma)$ is Koszul AS-regular of global dimension $4$  if $(\alpha,\beta,\gamma)$ is not equal to $(-1,+1,\gamma), (\alpha,-1,+1)$ or $(+1,\beta,-1)$. In the following, we always assume $(\alpha,\beta,\gamma)$ is not equal to these exceptional triples. 

Take $V = \bigoplus_{i=0}^3~ \kk x_i$ to be the generating space of $Skly_4(\alpha,\beta,\gamma)$  with  $V^* = \bigoplus_{i=0}^3 ~\kk\xi_i$.  We define the preregular $4$-linear form ${\sf s}_{\alpha \beta \gamma}: V^{*\otimes 4} \to \kk$ by the following values:
\smallskip

\[
\small{
\begin{array}{|l|l||l|l|}
\hline
 (i_0,i_1,i_2,i_3) & 
{\sf s}_{\alpha \beta \gamma}(\xi_{i_0},\xi_{i_1},\xi_{i_2},\xi_{i_3})&
 (i_0,i_1,i_2,i_3) & 
{\sf s}_{\alpha \beta \gamma}(\xi_{i_0},\xi_{i_1},\xi_{i_2},\xi_{i_3})\\
\hline 
(0,1,0,1) & 1 &
(1,0,1,0) & -1 \\
(0,2,0,2) & (1-\alpha)/(1+\beta) &
(2,0,2,0) & (\alpha-1)/(1+\beta) \\
(0,3,0,3) & (1+\alpha)/(1-\gamma)&
(3,0,3,0) & (1+\alpha)/(\gamma-1)\\
(1,2,1,2) & \gamma(1+\alpha)/(1-\gamma)&
(2,1,2,1) & \gamma(1+\alpha)/(\gamma-1)\\
(1,3,1,3)& \beta(\alpha-1)/(1+\beta)&
(3,1,3,1) & \beta(1-\alpha)/(1+\beta)\\
(2,3,2,3) & \alpha&
(3,2,3,2) & -\alpha\\
(0,1,2,3), (2,3,0,1) & -(1+\alpha)/2&
(1,2,3,0), (3,0,1,2) & (1+\alpha)/2\\
(0,1,3,2), (3,2,0,1)  & (1-\alpha)/2 &
(1,3,2,0), (2,0,1,3) & (\alpha-1)/2\\
(0,2,1,3), (1,3,0,2) & (1-\beta)(1-\alpha)/(2+2\beta)&
(2,1,3,0), (3,0,2,1) & (\beta-1)(1-\alpha)/(2+2\beta)\\
(0,2,3,1), (3,1,0,2)  & (\alpha-1)/2&
(2,3,1,0), (1,0,2,3) & (1-\alpha)/2\\
(0,3,1,2), (1,2,0,3)  & (1+\gamma)(1+\alpha)/(2\gamma-2) &
(3,1,2,0), (2,0,3,1) & (1+\gamma)(1+\alpha)/(2-2\gamma)\\
(0,3,2,1), (2,1,0,3)  & (1+\alpha)/2 &
(3,2,1,0),  (1,0,3,2) &-(1+\alpha)/2\\
\text{otherwise} & 0 & &\\
\hline
\end{array}
}
\]

\medskip

\begin{lemma}\label{L:Skly4}
We obtain that $Skly_4(\alpha,\beta,\gamma)$ is isomorphic to the superpotential algebra $A({\sf s}_{\alpha \beta \gamma},2)$, as algebras. \textnormal{(}Here, we abuse notation and let ${\sf s}_{\alpha \beta \gamma}$ denote both the  preregular  form and the corresponding twisted superpotential.\textnormal{)}
\end{lemma}

\begin{proof}
We use the same argument as in the proof of Lemma \ref{L:FAST}. According to \cite[Section~4.4]{SmithStafford}, the Koszul dual $S^!$ of the Sklyanin algebra $S:=Skly_4(\alpha,\beta,\gamma)$ is generated by $\{\xi_0,\xi_1,\xi_2,\xi_3\}$ with defining relations
\[
\begin{array}{lll}
2\xi_2\xi_3+(\alpha+1)\xi_0\xi_1-(\alpha-1)\xi_1\xi_0=0,&&
2\xi_3\xi_2+(\alpha-1)\xi_0\xi_1-(\alpha+1)\xi_1\xi_0=0,\\
2\xi_3\xi_1+(\beta+1)\xi_0\xi_2-(\beta-1)\xi_2\xi_0=0,&&
2\xi_1\xi_3+(\beta-1)\xi_0\xi_2-(\beta+1)\xi_2\xi_0=0,\\
2\xi_1\xi_2+(\gamma+1)\xi_0\xi_3-(\gamma-1)\xi_3\xi_0=0,&&
2\xi_2\xi_1+(\gamma-1)\xi_0\xi_3-(\gamma+1)\xi_3\xi_0=0,\\
\xi_0^2=\xi_1^2=\xi_2^2=\xi_3^2=0. &&
\end{array}
\]
Moreover, by \cite[Proposition 4.5]{SmithStafford}, $S_n^!=0$ for $n\ge 5$ and $S_4^!$ is spanned by $\xi_0\xi_1\xi_0\xi_1$.

By \cref{propA}, $Skly_4(\alpha,\beta,\gamma) \cong A({\sf s}, 2)$ for some twisted superpotential ${\sf s}$. In order to verify ${\sf s} = {\sf s}_{\alpha \beta \gamma}: V^{*\otimes 4}=S_1^{\otimes 4}\twoheadrightarrow S_4\cong \kk$, we apply the identities in the proof of \cite[Proposition 4.5]{SmithStafford}
\[
\begin{array}{lll}
\smallskip

\xi_0\xi_j\xi_0\xi_j=-\xi_j\xi_0\xi_j\xi_0\quad \text{for}\ 1\le j\le 3,\   &&\xi_0\xi_i\xi_0\xi_j=0\quad  \text{for}\ i\neq j,\\

 \xi_0\xi_3\xi_0\xi_3=\displaystyle\frac{1+\alpha}{1-\gamma} \xi_0\xi_1\xi_0\xi_1,
&&\xi_0\xi_2\xi_0\xi_2=\displaystyle\frac{1-\alpha}{1+\beta} \xi_0\xi_1\xi_0\xi_1.
\end{array}
\]
It is straightforward to see that we can write every $\xi_{i_0}\xi_{i_1}\xi_{i_2}\xi_{i_3}$ for $0\le i_0,i_1,i_2,i_3\le 3$ as a scalar multiple of $\xi_0\xi_1\xi_0\xi_1$ with the coefficient being the image of ${\sf s}_{\alpha \beta \gamma}$. This completes the proof.
\end{proof}

Now, the following proposition follows from Theorem \ref{H-Oisom} and Lemma \ref{L:Skly4}.

\begin{proposition}
We have that $\mathcal {O}_{Skly_4(\alpha',\beta',\gamma'),Skly_4(\alpha,\beta,\gamma)}(GL) \cong \mathcal H({\sf s}_{\alpha'\beta'\gamma'},{\sf s}_{\alpha\beta\gamma})$ as Hopf algebras. \qed
\end{proposition}

\subsection{On Yang-Mills algebras} Throughout the section, $\kk = \mathbb R$ and $s \in \mathbb{Z}_{\ge 1}$.
In \cite{ConnesDubois}, Connes and Dubois-Violette examined a cubic algebra generated by the covariant derivatives of a generic Yang-Mills connection over the $(s+1)$-dimensional pseudo-Euclidean space, which is given as follows.

\begin{definition}[$\mathcal{YM}(\mathbb G)$] Let $\mathbb G=(g_{ij})_{0\le i,j\le s}$ be an invertible real symmetric matrix. The {\it Yang-Mills algebra}  $\mathcal {YM}(\mathbb G)$ is generated by $x_0, x_1, \dots, x_s$  in degree one with $(s+1)$ cubic relations
$$
\textstyle \sum_{i,j=0}^s ~g_{ij}[x_i,[x_j,x_k]]=0, 
$$
for all $k\in \{0,\dots, s\}$. 
\end{definition}

It is shown in \cite[Theorem 1]{ConnesDubois} that the Yang-Mills algebra is $3$-Koszul and AS-regular of global dimension $3$. 
Moreover, the relations of $\mathcal {YM}(\mathbb G)$ can be rewritten as 
\begin{align}\label{E:YangMills2}
\textstyle \sum_{j,k,l=0}^s ~(g_{ij}g_{kl}+g_{il}g_{jk}-2g_{ik}g_{jl})x_jx_kx_l=0, 
\end{align}
for all $i\in \{0,\cdots,s\}$. We define the associated multilinear form ${\sf g}: (\mathbb R^{s+1})^{\otimes 4}\to \mathbb R$ such that 
\begin{align}\label{E:YMform}
{\sf g}_{ijkl}=g_{ij}g_{kl}+g_{il}g_{jk}-2g_{ik}g_{jl}
\end{align}
for $i,j,k,l\in \{0,\cdots, s\}$. The following lemma is straightforward.

\begin{lemma}
For any invertible real symmetric matrix $\mathbb G$, the Yang-Mills algebra $\mathcal{YM}(\mathbb G)$ and the superpotential algebra $A({\sf g},3)$ are isomorphic. \qed
\end{lemma}

Now the {\it Yang-Mills quantum group} can be understood by the result below, which  follows from Theorem \ref{H-Oisom}.

\begin{proposition}
We have $\mathcal O_{\mathcal{YM}(\mathbb G'),\mathcal{YM}(\mathbb G)}(GL) \cong \mathcal H({\sf g}',{\sf g})$ as Hopf algebras.  \qed
\end{proposition}

\section*{Acknowledgments}
A. Chirvasitu and C. Walton are both partially supported by the National Science Foundation by respective grants \#1565226 and \#1550306. X. Wang is partially supported with an AMS-Simons Travel grant.

\bibliography{QG-twopreregularforms-biblio}

\end{document}